\documentclass[11pt]{amsart}
\usepackage{amssymb}
\usepackage{amsmath}
\usepackage{amsfonts}
\usepackage{mathrsfs}
\usepackage{psfrag}
\usepackage{verbatim}
\usepackage{color}
\usepackage{bm}
\usepackage[usenames,dvipsnames]{xcolor}
\usepackage{dsfont}
\usepackage{overpic}
\usepackage{subcaption}
\usepackage{blindtext}

\usepackage[colorlinks=true]{hyperref}
\hypersetup{urlcolor=blue, citecolor=MidnightBlue}

\DeclareGraphicsExtensions{}
\theoremstyle{plain}
\newtheorem{theorem}{Theorem}[section]

\newtheorem{lemma}[theorem]{Lemma}
\newtheorem{proposition}[theorem]{Proposition}

\theoremstyle{definition}
\newtheorem{assumption}[theorem]{Assumption}

\theoremstyle{remark}
\newtheorem{remark}[theorem]{Remark}

\newcommand{\eps}{\varepsilon}

\newcommand{\kap}{\varkappa}

\newcommand{\BR}{\mathbb{R}}
\newcommand{\BP}{\mathbb{P}}
\newcommand{\BE}{\mathbb{E}}
\newcommand{\BZ}{\mathbb{Z}}
\newcommand{\BN}{\mathbb{N}}

\newcommand{\dd}{\mathsf{d}}
\newcommand{\hh}{\mathsf{h}}
\newcommand{\jj}{\mathsf{j}}

\newcommand{\ind}{\mathds{1}}

\begin{document}
\title[Asymptotic analysis of dynamical systems with sampling and jump noise]{Asymptotic analysis of dynamical systems driven by Poisson random measures with periodic sampling$^\dagger$
}\thanks{$\dagger$ An earlier version of this work was completed as a part of my Ph.D. thesis at the Department of Mathematics, Indian Institute of Technology
Gandhinagar, India.}

\author[Shivam Singh Dhama]{}
\address{Department of Mathematics and Statistics, Boston University, Boston, Massachusetts, 02215, usa}
\email{ssdhama@bu.edu, shivamsd.maths@gmail.com} 

 
\thanks{$^*$ Corresponding author}

\maketitle


\centerline{\scshape Shivam Singh Dhama$^*$}
\medskip
{\footnotesize
 
 \centerline{Department of Mathematics and Statistics, Boston University, Boston, Massachusetts, 02215, {\sc usa}}
} 
\begin{center}
\rule{17cm}{.01mm}
\end{center}
\begin{abstract}
In this article, we study the dynamics of a nonlinear system governed by an ordinary differential equation under the combined influence of fast periodic sampling with period $\delta$ and small \emph{jump} noise of size $\varepsilon, 0< \varepsilon,\delta \ll 1.$ The noise is a combination of Brownian motion and Poisson random measure. The instantaneous rate of change of the state depends not only on its current value \textit{but} on the most recent measurement of the state, as the state is measured at certain discrete-time instants. As $\varepsilon,\delta \searrow 0,$ the stochastic process of interest converges, in a suitable sense, to the dynamics of the deterministic equation. Next, the study of rescaled fluctuations of the stochastic process around its mean is found to vary depending on the relative rates of convergence of small parameters $\varepsilon, \delta$ in different asymptotic regimes. We show that the rescaled process converges, in a strong (path-wise) sense, to an effective process having an extra drift term capturing both the sampling and noise effect. Consequently, we obtain a first-order perturbation expansion of the stochastic process of interest, in terms of the effective process along with error bounds on the remainder.
\end{abstract}
\begin{center}
\rule{17cm}{.01mm}
\end{center}


\section{Introduction}\label{S:Intro}
In many control problems, it is often seen that the measurements (samples) of the state of the system of interest may not be available at all time instants. This issue motivates one to study a \textit{hybrid dynamical system} ({\sc hds}) with a given state-feedback control law. In this scenario, the dynamics of the system---described by an ordinary differential equation ({\sc ode})---evolves continuously with time, whereas the control inputs change via state samples at certain discrete-time instants. The control law remains constant between two discrete-time instants. Thus, the interaction of dynamics of the system at continuous and discrete-times makes it a {\sc hds}. For more details on {\sc hds}, we refer to \cite{GST-HybDynSys-book}. {\sc hds} appears very often in sampled-data-systems where the dynamics of system is controlled by a digital computer \cite{YuzGoodwin-book}. In these systems, a so-called \textit{sample-and-hold} implementation is frequently used. By this, we mean the samples of the state are measured at certain closely time instants and the control input is computed using state feedback control law; the control is then held fixed until the next sample of the state is taken.  For the explicit computation of time between samples, one can see \cite{NesicTeelCarnevale-TAC2009}. 

It is often noticed that the systems of interest in various applications are subjected to some small noises. The dynamics of such cases is described by stochastic differential equations ({\sc sde})\cite{applebaum2009levy,ikeda2014stochastic,KS91}. The behavior of dynamical systems perturbed by a small white noise has been extensively studied in \cite{FW_RPDS}, and references therein. Getting motivation from the asymptotic analysis presented in Chapter 2 of \cite{FW_RPDS}, we have explored this asymptotic behavior for general nonlinear systems in the presence of fast periodic sampling and small state-dependent white noise, see, \cite{dhama2022fluctuation}; for linear case, we refer to \cite{dhama2020approximation}. To the best of our knowledge, the dynamics of a class of nonlinear systems under the combined effects of sampling and noise with small {jumps} has not been explored yet. Thus, a natural question in the context of an extension of our sampling problem is: \emph{can one extend the results of our previous work \cite{dhama2020approximation,dhama2022fluctuation} for a significantly bigger class of nonlinear systems with the Poisson random measures case?} This paper answers this question affirmatively.

The goal of the present work is to study the asymptotic analysis of a nonlinear system governed by an {\sc ode} under the combined effect of periodic sampling with period $\delta$ and small noise of size $\eps, 0< \eps,\delta \ll 1.$ Here, noise is a combination of Brownian motion and Poisson random measures ({\sc prm}). The measurements (samples) of the state are taken every $\delta$ time units and the vector field of {\sc ode} depends on the current value of the state as well as its most recent measurement. We show that the dynamics of the resulting stochastic process indexed by two small parameters $\eps,\delta$ is close to the idealized deterministic system in the limit as $\eps,\delta \searrow 0$ over time interval $[0,T],$ for any fixed $T>0.$ Further, we explore the fluctuation analysis of the stochastic process of interest around its mean behavior and, interestingly, in the limiting {\sc sde} of the rescaled process, we get an extra \textit{effective drift} term. The limiting {\sc sde}, in fact, captures both the sampling and noise effects. Here, the fluctuation
behavior is found to vary, depending on the relative rates at which the two small parameters $\eps$ and $\delta$ approach zero. As a consequence, we get a first-order perturbation expansion for the stochastic process in powers of small parameter together with bounds on remainder. 

More precisely, in this paper, the {\sc sde} of our main interest is 
\begin{equation}\label{E:sde-intro}
dY^{\eps,\delta}_t = c\left(Y^{\eps,\delta}_{t-}, Y^{\eps,\delta}_{{\delta \lfloor t/\delta \rfloor-}}\right)dt + \eps \sigma(Y^{\eps,\delta}_{t-}) dW_t 
+ \eps  \int_{0<|x|<1}F(Y^{\eps,\delta}_{t-},x)\widetilde{N}(dt,dx), \quad Y^{\eps,\delta}_0=y_0,
\end{equation}
where the functions $c: \BR^n \times \BR^n \to \BR^n,$ $\sigma: \BR^n \to \BR^{n \times n}$ and $F: \BR^n \times \BR^n \to \BR^n,~n\in \BN$ are assumed to be measurable and satisfy a certain regularity condition. The processes $W$ and $\widetilde{N}$ represent a Brownian motion and compensated Poisson random measure, respectively and they are assumed independent to each other. $\lfloor \cdot \rfloor$ denotes the greatest integer function and $0< \eps,\delta \ll 1$. It is worth mentioning that equation \eqref{E:sde-intro} can be thought of as a random perturbation of a nonlinear control system $\dot{y}=c'(y,u),\thinspace c':\BR^n\times \BR^m \to \BR^n$ with a feedback control law $u=\kappa(y)$ via its sample-and-hold implementation; where by latter we mean the control function $u$ is updated at the time instants $k \delta,~k\in \BZ^+$ and it remains fixed in the sampling interval $(k\delta,(k+1)\delta).$ Therefore, in equation \eqref{E:sde-intro}, the sampling effect is through the term $Y^{\eps,\delta}_{{\delta \lfloor t/\delta \rfloor-}}$. Here, we also note that the dynamics of the deterministic counterpart of equation \eqref{E:sde-intro} (i.e.\footnote{Here, for any $t \ge 0$ and $x(t)\in \BR^n$, $x_t \triangleq x(t)$ and $\dot{x}_t \triangleq \frac{dx_t}{dt}$.}, $\dot{y}^{\delta}_t =  c(y_t^{\delta},y_{{\delta \lfloor t/\delta \rfloor}}^{\delta}), \thinspace y_0^\delta= y_0 \in \BR^n$) is fully non-linear. In our first result, viz., Theorem \ref{T:LLN}, we show that the stochastic process $Y_{t}^{\eps,\delta},$ $t\in [0,T]$ converges, as $\eps,\delta \searrow 0,$ to the dynamics $y_t$ which solves the deterministic equation $\dot{y}_t = c(y_t,y_t)$. Next, in our second main result (Theorem \ref{T:fluctuations-R-1-2-Levy})---which is the key contribution of this manuscript---we study the asymptotic analysis of the rescaled fluctuation process $Z^{\eps,\delta}_t \triangleq \eps^{-1}({Y^{\eps,\delta}_t - y_t})$ as $\eps,\delta$ vanish.
This result essentially depends on the interaction of the small parameters $\eps$ and $\delta.$ Following \cite{FreidlinSowers,KS2014}, we assume $\delta=\delta_{\eps}$ and $\lim_{\eps\searrow 0}\delta_\eps/\eps=\ell \in [0,\infty]$ exists, and the values $\ell=0$, $\ell \in (0,\infty)$ and $\ell=\infty$ correspond to Regimes 1, 2, and 3 respectively. We demonstrate that, in the first two regimes, the fluctuation process $Z_t^{\eps,\delta}$ can be approximated by an effective process $Z_t$ which solves the {\sc sde}
\begin{multline}\label{E:lim-fluct-R-1-2-Levy-intro}
Z_t= \int_0^t\{D_1 c(y_s,y_s)+D_2 c(y_s,y_s)\}Z_{s-} \thinspace ds +\frac{\ell}{2} \int_0^t D_2 c(y_s,y_s)\cdot c(y_s,y_s)\thinspace ds \\ + \int_0^t \sigma(y_s) dW_s
+\int_0^t\int_{0<|x|<1}F(y_s,x)\widetilde{N}(ds,dx).
\end{multline}
In this work, we use a strong (path-wise) notion  of convergence through calculating the explicit bounds for the terms $\BE\left[\sup_{0\le t \le T}|Y_t^{\eps,\delta}-y_t|^p\right],\thinspace p \in \{2,4\}$ and $\BE\left[\sup_{0 \le t \le T} |Y^{\eps,\delta}_t - y_t - \eps Z_t|^2\right]$ for any fixed $T>0$.

The work presented in this paper is novel due to two reasons: first, we consider here a significantly bigger class of nonlinear control systems $\dot{y}=c'(y,u),\thinspace c':\BR^n\times \BR^m \to \BR^n,\thinspace u=\kappa(y)$ under state-dependent jump noise effects. The  importance of this system is that it could be nonlinear in the control argument unlike \cite{dhama2022fluctuation} and the techniques here to study the asymptotic analysis are more wider (than in \cite{dhama2022fluctuation}) to accommodate the jump noise. Hence, this work can be thought of as a generalization of \cite{dhama2022fluctuation}, which studies the asymptotic analysis of a nonlinear control system $\dot{y}=f(y)+g(y)u,\thinspace u=\kappa(y)$ (affine in control) only under the white-noise effect. Second, we identify the extra drift term (i.e., $({\ell}/{2}) \int_0^t D_2 c(y_s,y_s)\cdot c(y_s,y_s) \thinspace ds$) in the limiting {\sc sde} \eqref{E:lim-fluct-R-1-2-Levy-intro} of the fluctuation process $Z_t^{\eps,\delta}$ for the jump noise case. The process of identification of this drift term while studying the fluctuations involves several intricate calculations. We identified this type of drift term---for different system classes---in our previous work \cite{dhama2020approximation,dhama2022fluctuation} also where the systems were perturbed \emph{only} by a Brownian noise; see \cite{dhama2020approximation} for the linear dynamics and \cite{dhama2022fluctuation} for a nonlinear case. Thus, this article extends the results of \cite{dhama2022fluctuation} for the case when the deterministic system is a general nonlinear system and noise is a combination of Brownian motion and {\sc prm}. 
 Another importance of this work is that we provide an approximation of a non-Markovian process (which is difficult to analyse in its own) by a Markov process. We acknowledge that some techniques in the proofs of the main results of this paper are similar to the article \cite{dhama2022fluctuation}.

There is a vast literature on the asymptotic study for the stochastic equations (driven by continuous or jump noise) with multiple small parameters. Many authors have carried out this analysis in the form of the averaging principle, which gives a simple dynamics of the slowly-varying component by averaging over the quickly-varying quantity. The classical work here dates back to Khasminskii \cite{has1966stochastic,khasminskij1968principle}. Since then a great amount of work has been developed on this topic. More recently, in various regimes depending on the rate of convergence of small parameters, many limit theorems for {\sc sde} driven by Brownian noise (without jumps) have been explored, see, for instance, \cite{AriAnupBor18,AthreyaBorkerSureshSundaresan19,
cerrai2009khasminskii,cerrai2009averaging,
feng2012small,FreidlinSowers,
pardoux2001poisson,pardoux2003poisson,pardoux2005poisson,
rockner2021averaging,rockner2021diffusion,KS2014}. These limit theorems can be interpreted as Law of Large Numbers ({\sc lln}) type result, Central Limit Theorem ({\sc clt}) type result and Large Deviation Principle ({\sc ldp}).

In many other research areas, for instance, in physics, insurance risk models, and kinetic chemical reactions, one encounters many complex stochastic systems with jump noise, for example, \cite{gillespie1976general,kulik2019non,liu2004asymptotic}. For certain optimal control problems with L{\'e}vy noise effect, one can see, e.g. \cite{do2019inverse,muthukumar2017infinite,
oksendal2014stochastic}. The averaging principle for stochastic models driven by Poisson random measures has been studied extensively by many authors, see, for example, \cite{givon2007strong,guo2018stochastic,
kulik2019non,mao2015averaging,xu2011averaging}, whereas for stochastic partial differential equations, we refer to \cite{debussche2013dynamics,pei2017two,
peszat2007stochastic,xu2017lp,xu2015strong}.
The interaction of {\sc ldp} with stochastic processes having jumps can be found in the recent articles \cite{budhiraja2013large,budhiraja2016moderate,
ado2022large,chen2017well,de1994large,de2000general,
rockner2007stochastic}. Finally, for an exit problem and metastable behavior of {\sc sde} driven by L{\'e}vy noise, see, e.g. \cite{imkeller2006first,imkeller2008metastable}.

We organize this paper as follows. In Section \ref{S:PS-MainResult-Levy}, we formulate our problem of interest described by {\sc sde} \eqref{E:sde} and state our main results, viz., Theorems \ref{T:LLN}, \ref{T:fluctuations-R-1-2-Levy} and \ref{T:Reg3-Ch4}. The proof of Theorem \ref{T:LLN} is given in Section \ref{SS:LLN-Proof} whereas the details of the proof of Theorem \ref{T:fluctuations-R-1-2-Levy} are spread out over Sections \ref{S:FCLT-Levy} and \ref{S:Main-Term-App-Levy}. Theorem \ref{T:fluctuations-R-1-2-Levy} is, indeed, proved in Section \ref{SS:Proof-CLT-Levy} via a couple of auxiliary results (Propositions \ref{P:Main-Term-App-Levy} through \ref{P:Poisson-Noise-Est})  whose proofs depend on a series of Lemmas \ref{L:M-terms-levy} through \ref{L:M4-Est-PoissonNoise}. Section \ref{S:Main-Term-App-Levy} is devoted to the proofs of these lemmas. At the end of this manuscript, in Section \ref{S:Conclusion-Levy}, we give some concluding remarks and future directions for the research.

\subsection*{Notation and conventions}
This section lists some of the notations used and conventions followed throughout this paper. The symbol $\triangleq$ is read ``is defined to equal.'' For $x\in \BR^n,$ $|x|\triangleq \sum_{i=1}^{n}|x_i|$ represents the one-norm on Euclidean space $\BR^n$, and for $M \in \BR^{n \times m}$, $|M|$ represents the induced matrix norm. The set $E$ will denote open unit ball centered and punctured at origin, i.e., 
$E\triangleq \{x\in \BR^n: 0<|x|<1\}$. The maximum and minimum of two real numbers $a$ and $b$ are written as $a \vee b$ and $a \wedge b$, respectively. We denote Borel $\sigma$-field on $\BR^n$ by $\mathscr{B}(\BR^n)$. For a given probability space $(\Omega, \mathscr{F}, \BP),$ the expectation operator $\BE[\cdot]$ represents the integration with respect to the probability measure $\BP$. For any $t\in [0,\infty)$ and $\delta>0,$ we define $\pi_\delta: [0,\infty) \to \delta \BZ^+$ by $\pi_\delta(t)\triangleq \delta \lfloor t/\delta \rfloor$. Here, the time-discretization operator $\pi_\delta(\cdot)$ rounds down the continuous time $t\in [0,\infty)$ to the nearest multiple of $\delta.$ Throughout this manuscript, $D_1c(x,y)\in \BR^{n \times n}$ and $D_2c(x,y)\in \BR^{n \times n}$ represent the Jacobian matrices of the mapping $c$ with respect to the variables $x$ and $y,$ respectively. We will be using the notation $``\lesssim"$ very frequently. For $a,b \in \BR,$ we write $a \lesssim b$ if there exists a positive constant $C$ such that $a \le C b;$ here, the value of constant $C$ may change from line to line and it will be allowed to depend on various parameters except only the small parameters $\eps$ and $\delta$.
However, in the statements of theorems, propositions and lemmas, the final constant $C$ will be denoted with subscript and depending on $T$ (e.g., for Theorem \ref{T:LLN}, it will be $C_{\ref{T:LLN}}(T)$). At many places in this paper, we will be using the following inequality: for $p>0,~n \in \BN$ and $x_1,\cdots,x_n\in \BR,$
\begin{equation}\label{E:Triangle-type-ineq-Levy}
(|x_1|+\cdots+|x_n|)^p \lesssim |x_1|^p+\cdots+|x_n|^p.
\end{equation}

A simple c\`{a}dl\`{a}g \footnote{ A function $f:[a,b]\to \BR$ is said to be c\`{a}dl\`{a}g on $[a,b]$ if, for all $s\in (a,b)$, $f$ is right-continuous and a left limit exists at $s$. For more details about c\`{a}dl\`{a}g functions, we refer to \cite{ConvProbMeas,EK86}.} functions property (mentioned below) will be frequently used in the whole manuscript: for $a,b\in \BR$ and a c\`{a}dl\`{a}g function $f:[a,b]\to \BR$, we have \cite[Section 2.9]{applebaum2009levy}
\begin{equation}\label{E:cadlag-property}
\sup_{a< x \le b}|f(x-)|\le \sup_{a\le x \le b}|f(x)|.
\end{equation}

\subsection*{Preliminaries}
In this section, we present some basic definitions and results from \cite{applebaum2009levy}, which will be quite useful in the next sections. A few more references relevant to this section are \cite{ken1999levy,protter2005stochastic,
ikeda2014stochastic,kunita2004stochastic}.

Let $(\Omega,\mathscr{F}, \BP)$ be a probability space on which a L{\'e}vy process $U=\{U_t: t\ge 0\}$ taking values in $\BR^n,\thinspace n\in \BN$ is defined.
For any $t \ge 0$ and $M \in \mathscr{B}(\BR^n-\{0\})$, we define the process $N(t,M)\triangleq \sum_{0 \le s \le t} \ind_M (\Delta U_s),$ where $\Delta U_t \triangleq U_t-U_{t-}.$ We see from here that for any fixed $t \ge 0 $ and $\omega \in \Omega,$ $N(t,\cdot)(\omega)$ is a counting measure on $\mathscr{B}(\BR^n-\{0\}).$ The measure defined by $\nu(\cdot)\triangleq \BE[N(1,\cdot)]$ is said to be the intensity measure associated with the L{\'e}vy process $U$. Indeed, measure $\nu$ satisfies  
$\int_{\BR^n-\{0\}}(|x|^2 \wedge 1)\nu(dx)< \infty$ and is known as the L{\'e}vy measure. A set $M\in \mathscr{B}(\BR^n-\{0\})$ is said to be bounded below if $0$ is not an element of $\overline{M},$ where $\overline{M}$ denotes the closure of $M.$ For each $t\ge 0$ and a bounded below set $M$, $N(t,M)$ is a Poisson random measure ({\sc prm}).
For more details about {\sc prm}, one can see \cite[Section 2.3.1]{applebaum2009levy}. We define the compensated Poisson random measure as follows:
\begin{equation*}
\widetilde{N}(t,M)\triangleq N(t,M)-t\nu(M).
\end{equation*} 

For any fixed $T>0,$ let $\mathcal{H}_2(T,E)$ denotes the collection of all functions $f: [0,T]\times E \times \Omega \to \BR$ which are predictable and $\int_0^T\int_E \BE(|f(t,x)|^2)dt \thinspace \nu(dx) < \infty.$ Under these conditions, the integral $\int_0^T\int_E f(t,x)\widetilde{N}(dt,dx)$ is well-defined, see, \cite[Section 4.2]{applebaum2009levy} for more details. For integral \\ $\int_0^T\int_E f(t,x)\widetilde{N}(dt,dx)$, the identity
 $$\BE\left[\left|\int_0^T\int_E f(t,x)\widetilde{N}(dt,dx)\right|^2\right]= \BE\left[\int_0^T \int_E|f(t,x)|^2dt \thinspace \nu(dx) \right]$$
 is called It\^o isometry. We will be using this equation very frequently in our calculations to handle the stochastic integral with respect to {\sc prm}.



\section{Problem Statement and Main Results}\label{S:PS-MainResult-Levy}
In many control problems, specifically, when a system is controlled by a digital computer, one considers a general nonlinear hybrid equation\footnote{One obtains this equation by using the sample-and-hold implementation of control law $u=\kappa(y)$ \cite{Hes_LST} in the differential equation $\dot{y}_t=c'(y_t,u)$ with initial vector $y_0\in \BR^n.$ Here, the state of the system is measured at times $k\delta,~k\in \BZ^+.$ The dynamics of this system in the sampling interval $[k\delta,(k+1)\delta)$ is given by $\dot{y}_t^\delta=c({y}_t^\delta, {y}_{k \delta}^\delta)$, with ${y}_{k \delta-}^\delta={y}_{k \delta}^\delta$ and $y^\delta_{0-} \triangleq y^\delta_{0}=y_0.$ For more details about control systems, we refer to \cite{oksendal2005stochastic,Zabczyk-MCT}.}  
\begin{equation}\label{E:ie-sampling}
\dot{y}^{\delta}_t =  c\left(y_t^{\delta}, y_{\pi_\delta(t)}^{\delta}\right), \quad y_0^\delta= y_0 \in \BR^n,
\end{equation}
where, ${y}^{\delta}_t: [0,T] \to \BR^n$ represents the state of the system.
Here, the measurements of state $y^\delta$ are taken every $\delta$ time units and the instantaneous value of $\dot{y}^\delta_t$ depends not only on $y_t^\delta$ but also on the most recent sample of $y^\delta,$ viz., $y_{\pi_\delta(t)}^{\delta}.$ One can easily see that the first component (i.e., $y_t^\delta$) on the right side of equation \eqref{E:ie-sampling} changes continuously with time; in contrast to that, the second one (i.e., $ y_{\pi_\delta(t)}^{\delta}$) changes only at the sampling time instants. This interaction of dynamics at continuous and discrete-time instants leads to a hybrid dynamical system. As $\delta \searrow 0$, one expects that the dynamics of $y^{\delta}_t$ converges to the behavior of $y_t$ solving the {\sc ode} 
\begin{equation}\label{E:closed-system}
\dot{y}_t = c(y_t,y_t), 
\end{equation}
with initial condition $y_0\in \BR^n.$ This is, indeed, shown in Proposition \ref{P:LLN-Det-Approx-Levy}.

We would now like to explore the situation when the dynamics of system \eqref{E:ie-sampling} is perturbed by a small noise which is a combination of Brownian motion and {\sc prm}. Throughout this paper, we will be assuming that the Brownian motion $W$ and {\sc prm} $N$ are independent and are defined on the same probability space $(\Omega,\mathscr{F}, \BP)$. More precisely, we study the asymptotic behavior of the stochastic process $Y^{\eps,\delta}_t$ solving {\sc sde}
\begin{equation}\label{E:sde}
dY^{\eps,\delta}_t = c\left(Y^{\eps,\delta}_{t-}, Y^{\eps,\delta}_{{\pi_\delta(t)-}}\right)dt + \eps \sigma(Y^{\eps,\delta}_{t-}) dW_t\\ 
+ \eps  \int_{E}F(Y^{\eps,\delta}_{t-},x)\widetilde{N}(dt,dx), \quad Y^{\eps,\delta}_0=y_0,
\end{equation}
in the limit as $\eps,\delta \searrow 0$. Recall that the parameters $\eps$ and $\delta$ correspond to the size of noise and time between the samples of the state. The mappings $c: \BR^n \times \BR^n \to \BR^n,\thinspace \sigma: \BR^n \to \BR^{n \times n}$ and $F: \BR^n \times \BR^n \to \BR^n$ are assumed to be measurable and further conditions are given in Assumptions \ref{A:Lip-continuity-Levy} and \ref{A:Linear-growth-condition-Levy} below. The equation \eqref{E:sde} is, of course, shorthand of the integral equation 
$$Y^{\eps,\delta}_t = y_0 + \int_0^t c\left(Y^{\eps,\delta}_{s-}, Y^{\eps,\delta}_{{\pi_\delta(s)-}}\right)ds + \eps \int_0^t \sigma(Y^{\eps,\delta}_{s-}) dW_s 
+ \eps \int_0^t \int_{E}F(Y^{\eps,\delta}_{s-},x)\widetilde{N}(ds,dx).$$
The following conditions are imposed on the functions $c,$ $\sigma$ and $F$.
\begin{assumption}[Lipschitz continuity]\label{A:Lip-continuity-Levy}
There exists a positive constant $\tilde{C}$ such that for any $x_1, x_2, z_1, z_2 \in \BR^n,$ we have
\begin{equation*}
\begin{aligned}
|c(x_1,x_2)-c(z_1,z_2)|^2 & \le \tilde{C}(|x_1-z_1|^2+|x_2-z_2|^2),\\
|\sigma(x_1)-\sigma(x_2)|^2 + \int_{E}|F(x_1,x)-F(x_2,x)|^2 \thinspace \nu(dx) &\le \tilde{C} |x_1-x_2|^2.
\end{aligned}
\end{equation*}
\end{assumption}

\begin{assumption}[Linear growth condition]\label{A:Linear-growth-condition-Levy}
There exists a positive constant $\tilde{C}$ such that for any $z \in \BR^n,$ we have
\begin{equation*}
|\sigma(z)|^4 + \int_{E}|F(z,x)|^4 \thinspace \nu(dx)\le \tilde{C}(1+|z|^4).
\end{equation*}
\end{assumption}

\begin{assumption}[Boundedness/linear growth of derivatives]\label{A:Derivative}
For the vectors $x=(x_1,\cdots, x_n),\thinspace y=(y_1,\cdots,y_n)\in \BR^n$ and $c:\BR^n \times \BR^n \to \BR^n,$ we have
\begin{equation*}
\begin{aligned}
\left|\frac{\partial c_i}{\partial x_k}(x,y)\right|&\le C(1+|y|),\quad \left|\frac{\partial c_i}{\partial y_k}(x,y)\right| \le C, \quad i,j,k \in \{1, \cdots, n\}, \\
 \left|\frac{\partial^2 c_i}{\partial x_k \partial y_j}(x,y)\right|&\le C,\quad \left|\frac{\partial^2 c_i}{\partial x_k \partial x_j}(x,y)\right| \le C(1+|y|), \quad and \quad \left|\frac{\partial^2 c_i}{\partial y_k \partial y_j}(x,y)\right|\le C.
 \end{aligned}
\end{equation*}
\end{assumption}

\begin{remark}
Throughout this work, Assumption \ref{A:Derivative} is important and is of our particular interest as it allows one to cover the typical control system $\dot{y}_t = f(y_t)+g(y_t)\kappa(y_{\pi_\delta(t)})$ (studied in \cite{dhama2022fluctuation}) which is simply a particular case of model \eqref{E:closed-system} with $c(x,y)=f(x)+g(x)\kappa(y); f:\BR^n \to \BR^n, g:\BR^n \to \BR^{n \times m}, \kappa:\BR^n \to \BR^m $ are some regular maps. It is worth mentioning that Assumption \ref{A:Derivative} covers the conditions imposed on the mappings $f,g, \kappa$ (where $f,\kappa$ grow linearly and $g$ is bounded) in \cite{dhama2022fluctuation}. 
\end{remark} 

\begin{remark}[Comment on the existence and uniqueness result]
We can easily show that under Assumptions \ref{A:Lip-continuity-Levy} and \ref{A:Linear-growth-condition-Levy} and following the similar calculations presented in \cite[Theorem 6.2.3]{applebaum2009levy}, which uses \textit{Picard iteration} technique in the proof, equation \eqref{E:sde} has a unique strong solution.
\end{remark}

We are now ready to state our first main result.

\begin{theorem}(Law of Large Numbers Type Result)\label{T:LLN}
Let $y_t$ and $Y_t^{\varepsilon,\delta}$ be the solutions of equations \eqref{E:closed-system} and \eqref{E:sde}, respectively.
Then, for any fixed $T>0$, and $p \in \{2,4\},$ there exists a positive constant $C_{\ref{T:LLN}}(T)$ such that 
for any $\eps,\delta>0$
\begin{equation*}
\BE\left[\sup_{0\le t \le T}|Y_t^{\eps,\delta}-y_t|^p\right]\le 
(\eps^p + \delta^p)C_{\ref{T:LLN}}(T).
\end{equation*}
\end{theorem}

This result can be interpreted as a {\sc lln} type result and shows that the dynamics of the process $Y_t^{\eps,\delta}$ is close to the behavior of deterministic system described by $y_t$ on time interval $[0,T],$ for any fixed $T>0$ and the error is of order $\mathscr{O}(\eps \vee \delta).$

Next, we explore the fluctuation analysis of the process $Y_t^{\eps,\delta}$ around its mean $y_t.$ This enables us to write the process $Y_t^{\eps,\delta}$ as a first order perturbation expansion in powers of the small parameter together with estimates on remainder. Interestingly, the fluctuation study of process $Y_t^{\eps,\delta}$ is found to vary depending on the relative rates of convergence of small parameters $\eps$ and $\delta$ to zero. 

To make things precise, we assume $\delta=\delta_{\eps}$ and 
$\ell \triangleq \lim_{\eps \searrow 0}{\delta_{\eps}}/\eps$ exists in $[0,\infty].$ Following the articles \cite{FreidlinSowers,KS2014}, we consider the following three different regimes\footnote{In \cite{FreidlinSowers}, the parameter $\delta$ corresponds to the homogenization.} :
\begin{equation}\label{E:cc-Levy}
\ell \triangleq \lim_{\eps \searrow 0}\delta_\eps/\eps \begin{cases}=0 & \text{Regime 1,}\\ \in (0,\infty) & \text{Regime 2,}\\ = \infty & \text{Regime 3.}\end{cases}
\end{equation}
In Regimes 1 and 2, we define the rescaled fluctuation process
\begin{equation}\label{E:fluct-processes-Levy}
Z^{\eps,\delta}_t \triangleq \frac{Y^{\eps,\delta}_t - y_t}{\eps}. 
\end{equation}
Here, we note that the coarser parameter $\eps$ is used to rescale the stochastic quantity $(Y_t^{\eps,\delta}-y_t).$ To get more insight into the rescaled process $Z_t^{\eps,\delta},$ we recall equations \eqref{E:closed-system} and \eqref{E:sde} to get
\begin{equation*}
Z^{\eps,\delta}_t=(1/\eps)\int_0^t \left\{c\left(Y^{\eps,\delta}_{s-}, Y^{\eps,\delta}_{{\pi_\delta(s)-}}\right)-c(y_s,y_s)\right\}\thinspace ds + \int_0^t \sigma(Y^{\eps,\delta}_{s-})dW_s \\
 + \int_0^t\int_{E}F(Y^{\eps,\delta}_{s-},x)\widetilde{N}(ds,dx).
\end{equation*}
We apply Taylor's theorem to get
\begin{multline}\label{E:Central-Limit-Eq-Levy}
Z^{\eps,\delta}_t=\int_0^t \left\{D_1c(y_s,y_s)+D_2c(y_s,y_s)\right\}Z^{\eps,\delta}_{s-}\thinspace ds + \int_0^t D_2c(y_s,y_s)\left(\frac{Y^{\eps,\delta}_{{\pi_\delta(s)-}}-Y_{s-}^{\eps,\delta}}{\eps} \right)ds\\
+ \int_0^t \sigma(Y^{\eps,\delta}_{s-})dW_s  + \int_0^t\int_{E}F(Y^{\eps,\delta}_{s-},x)\widetilde{N}(ds,dx) + {\sf R}_t^{\eps,\delta}, \quad \text{where}
\end{multline}

\begin{multline*}
{\sf R}_t^{\eps,\delta}\triangleq \int_0^t \left[\frac{c\left(Y^{\eps,\delta}_{s-}, Y^{\eps,\delta}_{{\pi_\delta(s)-}}\right)-c(y_s,y_s)}{\eps} - D_1c(y_s,y_s)Z^{\eps,\delta}_{s-} - D_2c(y_s,y_s)Z^{\eps,\delta}_{s-} \right. \\
\left. -D_2c(y_s,y_s)\left(\frac{Y^{\eps,\delta}_{{\pi_\delta(s)-}}-Y_{s-}^{\eps,\delta}}{\eps} \right)\right]ds. 
\end{multline*}
The above equation can be written in the following way as the remainder term will be shown small later in an appropriate way.

\begin{multline*}
Z^{\eps,\delta}_t=\int_0^t \left\{D_1c(y_s,y_s)+D_2c(y_s,y_s)\right\}Z^{\eps,\delta}_{s-}\thinspace ds + \int_0^t D_2c(y_s,y_s)\left(\frac{Y^{\eps,\delta}_{{\pi_\delta(s)-}}-Y_{s-}^{\eps,\delta}}{\eps} \right)ds\\
+ \int_0^t \sigma(Y^{\eps,\delta}_{s-})dW_s  + \int_0^t\int_{E}F(Y^{\eps,\delta}_{s-},x)\widetilde{N}(ds,dx) + \mathcal{O}(\eps),
\end{multline*}


We now aim to characterize the limiting process (denoted by $Z_t$), as $\eps,\delta$ tend to zero, of the fluctuation process $Z_t^{\eps,\delta}$ satisfying {\sc sde} \eqref{E:Central-Limit-Eq-Levy}. A bit informally, the limiting {\sc sde} \eqref{E:lim-fluct-R-1-2-Levy} below, is obtained from equation \eqref{E:Central-Limit-Eq-Levy} by replacing $Z_t^{\eps,\delta}$ by $Z_t, $ ${\eps^{-1}}\int_0^t D_2c(y_s,y_s)\cdot(Y_{s-}^{\varepsilon, \delta}-Y_{\pi_\delta(s)-}^{\varepsilon, \delta})\thinspace ds$ by $\dd(t)\footnote{The identification of $\dd(t)$ is given in equation \eqref{E:ell-Levy}.}$, and $Y_{t-}^{\eps,\delta}$ by $y_t$ in the stochastic integrals $\int_0^t \sigma(Y^{\eps,\delta}_{s-})dW_s,$ and $\int_0^t\int_{E}F(Y^{\eps,\delta}_{s-},x)\widetilde{N}(ds,dx).$ Precise estimates for the terms ${\eps^{-1}}\int_0^tD_2c(y_s,y_s)\cdot (Y_{s-}^{\varepsilon, \delta}-Y_{\pi_\delta(s)-}^{\varepsilon, \delta})\thinspace ds-\dd(t), ~\int_0^t \{\sigma(Y^{\eps,\delta}_{s-})-\sigma(y_s)\}dW_s$ and \\ $\int_0^t\int_{E}\{F(Y^{\eps,\delta}_{s-},x)-F(y_s,x)\}\widetilde{N}(ds,dx)$ are obtained in Propositions \ref{P:Main-Term-App-Levy} through \ref{P:Poisson-Noise-Est} in Section \ref{S:FCLT-Levy}.


Before stating our second main result, we fix some notations. Recall $\delta=\delta(\eps)$ and in Regimes 1 and 2, $\lim_{\eps \searrow 0}\delta_\eps/\eps = \ell \in [0,\infty)$. Consequently, there exists $\eps_0 \in (0,1)$ such that $\left|\frac{\delta_{\eps}}{\eps}-\ell\right|<1, \text{whenever}, 0<\eps<\eps_0.$ In particular, for $0<\eps<\eps_0$, we have 
\begin{equation}\label{E:eps0}
\delta_{\eps}<(\ell+1)\eps.
\end{equation}

For the cases $\ell=0$, $\ell \in (0,\infty)$, set
\begin{equation}\label{E:kappa-Levy}
\kap(\eps) \triangleq \left|\frac{\delta}{\eps}-\ell\right|.
\end{equation}
Of course, $\lim_{\eps \searrow 0} \kap(\eps)=0$. 

We now state our second main result which can be interpreted as a {\sc clt} type result.

\begin{theorem}(Central Limit Theorem Type Result)\label{T:fluctuations-R-1-2-Levy}
Let $y_t$ and $Y_t^{\varepsilon,\delta}$ solve \eqref{E:closed-system} and \eqref{E:sde}, respectively. Suppose that we are in Regime $i \in \{1,2\}$, i.e., $\lim_{\eps \searrow 0}\delta_\eps/\eps = \ell \in [0,\infty)$. 
Let $Z=\{Z_t: t \ge 0\}$ be the unique solution of 
\begin{multline}\label{E:lim-fluct-R-1-2-Levy}
Z_t= \int_0^t\{D_1c(y_s,y_s)+D_2c(y_s,y_s)\} Z_{s-} \thinspace ds +\frac{\ell}{2} \int_0^t D_2c(y_s,y_s)\cdot c(y_s,y_s) \thinspace ds \\ + \int_0^t \sigma(y_s) dW_s
+\int_0^t\int_{E}F(y_s,x)\widetilde{N}(ds,dx).
\end{multline}
Then, for any fixed $T \in (0,\infty)$, there exists a positive constant $C_{\ref{T:fluctuations-R-1-2-Levy}}(T)$ such that for $0<\eps<\eps_0$, we have
\begin{equation}\label{E:FCLT}
\BE\left[\sup_{0 \le t \le T} |Z^{\eps,\delta}_t - Z_t|^2\right] = 
\frac{1}{\eps^2}\BE\left[\sup_{0 \le t \le T} |Y^{\eps,\delta}_t - y_t - \eps Z_t|^2\right]  \le  (\ell+1)^2\left[\eps^2+{\delta}+\varkappa^2(\eps)\right]C_{\ref{T:fluctuations-R-1-2-Levy}}(T),
\end{equation}
where $\eps_0$ and $\kap(\eps) \searrow 0$ are as in equations \eqref{E:eps0} and \eqref{E:kappa-Levy}, respectively.
\end{theorem}

\begin{remark}[Interpretation]
The interpretation of Theorem \ref{T:fluctuations-R-1-2-Levy} is twofold. First, the paths of process $Z_t^{\eps,\delta}$ converge to the paths of process $Z_t$ over time interval $[0,T],$ for any fixed $T>0$ and the error is of order $\mathscr{O}(\sqrt{\eps} \thinspace \vee \varkappa(\eps)).$ Second, the expression $\BE\left[\sup_{0 \le t \le T} |Y^{\eps,\delta}_t - y_t - \eps Z_t|^2\right]$ converges to zero faster than $\eps^2$ does. This implies $\BE\left[\sup_{0 \le t \le T} |Y^{\eps,\delta}_t - y_t - \eps Z_t|^2\right]=o(\eps^2);$ thus informally, 
$Y_t^{\eps,\delta}=y_t+ \eps Z_t+ o(\eps).$ Here, we not only characterize the coefficient of $\eps$, viz., the process $Z_t$ in equation \eqref{E:lim-fluct-R-1-2-Levy} but, also give error estimates on the remainder.
\end{remark}

\begin{remark}[Solution representation]\label{R:solu-reg-2}
Being a linear {\sc sde} in $Z_t$ variable (and an Ornstein-Uhlenbeck type equation), the equation \eqref{E:lim-fluct-R-1-2-Levy} can be solved explicitly for the process $Z_t$. Indeed, a simple application of It\^{o}'s formula \cite[Theorem 4.4.7]{applebaum2009levy} to the function $\varphi(t,x)=x \thinspace e^{-\int_0^t \sum_{i=1}^2 D_i c(y_s, y_s)\thinspace ds}$ gives  
\begin{multline*}
Z_t= \frac{\ell}{2} \int_0^t e^{\int_s^t \sum_{i=1}^2 D_i c(y_u, y_u)\thinspace du}D_2c(y_s,y_s)\cdot c(y_s,y_s) \thinspace ds \\
 + \int_0^t e^{\int_s^t \sum_{i=1}^2 D_i c(y_u, y_u)\thinspace du} \sigma(y_s)\thinspace dW_s
+\int_0^t\int_{E}e^{\int_s^t \sum_{i=1}^2 D_i c(y_u, y_u)\thinspace du}F(y_s,x)\widetilde{N}(ds,dx).
\end{multline*}
Further, in a particular case when $c(y,y)=(A-BK)y,$ where, $A\in \BR^{n \times n}\thinspace, B\in \BR^{n \times m}$, $K \in \BR^{m \times n}$ and if the matrices $A$ and $BK$ commute, i.e., $A(BK)=(BK)A,$ then the above equation can be simplified using the solution representation $y_s=e^{(A-BK)s}y_0$ for the deterministic system $\dot{y}_s = (A-BK)y_s$ to yield
\begin{equation*}
Z_t= \frac{\ell}{2}BKt(A-BK)y_t + \int_0^t e^{(A-BK)(t-s)} \sigma(y_s)\thinspace dW_s
+\int_0^t\int_{E}e^{(A-BK)(t-s)}F(y_s,x)\widetilde{N}(ds,dx).
\end{equation*} 
\end{remark}

\begin{remark}[A comparison with our previous work]
For the particular choices of $c(y,y)=(A-BK)y$, $\sigma=I$ (where, $I$ represents the $n \times n$ identity matrix and $A,B,K$ are also some matrices of suitable dimensions) and $F=0$ in equation \eqref{E:sde}, this problem turns out to be a special case of the problem studied in \cite{dhama2020approximation}. In contrast to \cite{dhama2020approximation}, due to the nonlinearity of our model, the explicit solution representation of the state is not available here. We studied an another particular case of the present model in \cite{dhama2022fluctuation} where we consider $c(x,y)=f(x)+g(x)\kappa(y)$, and $F=0$ (where, $f,g,\kappa$ are some suitable mappings). In all these works, our key finding is the extra effective drift term; however, due to our model choices, the drift term takes a different form in the above mentioned works. Hence, the part of innovation of the present work is to conduct the asymptotic analysis and identifying the drift term for a general nonlinear sampled system with state-dependent Poisson random measures case.

\end{remark}


In case of Regime 3, we notice that $\delta$ is the coarser parameter (see, equation \eqref{E:cc-Levy}). Therefore, rescaling the quantity $(Y_t^{\eps,\delta}-y_t)$ by the parameter $\delta$ and following similar calculations for Regimes 1 and 2, we can get the fluctuation analysis for the process $V_t^{\eps,\delta}\triangleq {\delta}^{-1}(Y_t^{\eps,\delta}-y_t).$ Before stating our main result for Regime 3, we note
\begin{equation}\label{E:R3}
\tilde\kap(\delta)  \triangleq {\eps}/{\delta} \searrow 0 \quad \text{as} \quad \delta \searrow 0,
\end{equation}
there exists $\delta_0\in(0,1)$ such that whenever $0<\delta<\delta_0,$ we have, $\eps <\delta$.


\begin{theorem}\label{T:Reg3-Ch4}
Let $y_t$ and $Y_t^{\varepsilon,\delta}$ solve \eqref{E:closed-system} and \eqref{E:sde}, respectively. Suppose that we are in Regime 3, i.e., $\ell=\infty.$ Let $V=\{V_t:t\ge 0\}$ be the unique solution of 
\begin{equation}\label{E:FCLT-R3}
V_t= \int_0^t\{D_1c(y_s,y_s)+D_2c(y_s,y_s)\} V_{s} \thinspace ds +\frac{1}{2} \int_0^t D_2c(y_s,y_s)\cdot c(y_s,y_s) \thinspace ds.
\end{equation}
Then, for any fixed $T>0,$ there exists a positive constant $C_{\ref{T:Reg3-Ch4}}(T)$ such that for $0<\delta<\delta_0,$ we have
\begin{equation*}
\BE\left[\sup_{0 \le t \le T} |V^{\eps,\delta}_t - V_t|^2\right] = 
\frac{1}{\delta^2}\BE\left[\sup_{0 \le t \le T} |Y^{\eps,\delta}_t - y_t - \delta V_t|^2\right] \le  \left[{\delta}+\tilde{\varkappa}^2(\delta)\right]C_{\ref{T:Reg3-Ch4}}(T),
\end{equation*}
where $\delta_0$ and $\tilde{\varkappa}(\delta) \searrow 0$ are as in equation \eqref{E:R3}.
\end{theorem}

We also note here that equation \eqref{E:FCLT-R3} is \emph{deterministic} as noise parameter $\eps$ vanishes faster than the sampling parameter $\delta$ does. For the sake of brevity, we avoid the calculations for Regime 3.
\begin{remark}[Solution representation]
By the method of variation of parameters, the limiting integral equation \eqref{E:FCLT-R3} can be solved easily for $V_t$ to give $V_t=({1}/{2}) \int_0^t e^{\int_s^t \sum_{i=1}^2 D_i c(y_u, y_u)\thinspace du}D_2c(y_s,y_s)\cdot c(y_s,y_s) \thinspace ds$. This expression can further be simplified for a particular case as we did in Remark \ref{R:solu-reg-2}.
\end{remark}

\section{Limiting mean behavior}\label{S:LLN}

In this section, we prove our first main result (Theorem \ref{T:LLN}) through Propositions \ref{P:LLN-Stoc-App-Levy} and \ref{P:LLN-Det-Approx-Levy} stated below. Proposition \ref{P:LLN-Stoc-App-Levy} deals with the stochastic term $\sup_{0\le t \le T}|Y_t^{\eps,\delta}-y_t^\delta|^p, \thinspace p \in \{2,4\}$ in a path-wise sense whereas Proposition \ref{P:LLN-Det-Approx-Levy} gives an estimate for the deterministic quantity $\sup_{0\le t \le T} |y_t^\delta-y_t|^p, \thinspace p \in \BN$. It is shown in these propositions that the terms of interest are of order  $\mathscr{O}(\eps \vee \delta).$ 

\subsection{Proof of Theorem \ref{T:LLN}}\label{SS:LLN-Proof}
We start by stating Propositions \ref{P:LLN-Stoc-App-Levy} and \ref{P:LLN-Det-Approx-Levy} which are the principal components in the proof of Theorem \ref{T:LLN}.
\begin{proposition}\label{P:LLN-Stoc-App-Levy}
Let $y_t^\delta$ and $Y_t^{\eps,\delta}$ be the solutions of \eqref{E:ie-sampling} and \eqref{E:sde}, respectively. Then, for any fixed $T>0$ and $p \in \{2,4\}$, there exists a positive constant $C_{\ref{P:LLN-Stoc-App-Levy}}(T)$ 
such that for any $\eps, \delta>0$, we have
\begin{equation*}
\BE\left[ \sup_{0\le t \le T}|Y_t^{\eps,\delta}-y_t^\delta|^p \right] \le \eps^p C_{\ref{P:LLN-Stoc-App-Levy}}(T).
\end{equation*}
\end{proposition}

The proof of Proposition \ref{P:LLN-Stoc-App-Levy} is deferred to the end of this section.

\begin{proposition}\label{P:LLN-Det-Approx-Levy}
Let $y_t^\delta$ and $y_t$  be the solutions of \eqref{E:ie-sampling} and \eqref{E:closed-system}, respectively. Then, for any fixed $T>0,$ and $p \in \BN$, there exists a positive constant $C_{\ref{P:LLN-Det-Approx-Levy}}(T)$ 
such that for any $\eps,\delta>0$, we have 
\begin{equation*}
\sup_{0\le t \le T} |y_t^\delta-y_t|^p \le \delta^p C_{\ref{P:LLN-Det-Approx-Levy}}(T).
\end{equation*}
\end{proposition}
\begin{proof}[Proof of Proposition \ref{P:LLN-Det-Approx-Levy}]
From the integral representations of $y_t$ and $y_t^\delta,$ we have
\begin{equation*}
|y_t^\delta-y_t|^p \le \left(\int_0^t \left|c(y_s^\delta,y_{\pi_\delta(s)}^\delta)-c(y_s,y_s) \right|ds \right)^p.
\end{equation*}
Now, a simple application of the Lipschitz continuity of $c$, H$\ddot{\text{o}}$lder's inequality, Lemma \ref{L:Sampling-Difference-Levy} and finally Gronwall's inequality yields the required result. 
\end{proof}


We now prove Theorem \ref{T:LLN}.

\begin{proof}[Proof of Theorem \ref{T:LLN}] 
For any $p \in \{2,4\},$ writing $Y_t^{\eps,\delta}-y_t$ as $Y_t^{\eps,\delta}-y_t^\delta+ y_t^\delta -y_t$, we get
\begin{equation}\label{E:LLN-ineq-Levy}
\begin{aligned}
|Y_t^{\eps,\delta}-y_t|^p \le \left[|Y_t^{\eps,\delta}-y_t^{\delta}|+|y_t^{\delta}-y_t|\right]^p
\lesssim |Y_t^{\eps,\delta}-y_t^{\delta}|^p + |y_t^{\delta}-y_t|^p,
\end{aligned}
\end{equation}
where the last expression in the above equation is obtained by using the inequality \eqref{E:Triangle-type-ineq-Levy}. Now, taking supremum over $[0,T]$ on both sides of equation \eqref{E:LLN-ineq-Levy}, we have 
\begin{equation*}
\sup_{0\le t \le T}|Y_t^{\eps,\delta}-y_t|^p \lesssim \sup_{0\le t \le T}|Y_t^{\eps,\delta}-y_t^{\delta}|^p + \sup_{0 \le t \le T}|y_t^{\delta}-y_t|^p.
\end{equation*}
The proof is now completed using Propositions \ref{P:LLN-Stoc-App-Levy} and \ref{P:LLN-Det-Approx-Levy}.
\end{proof}


The proof of Proposition \ref{P:LLN-Stoc-App-Levy} requires estimates for the deterministic quantities  $\sup_{0\le t \le T}|y_t^\delta|^p, \\ 
 \sup_{0\le t \le T}|y_t|^p, |y_t-y_{\pi_{\delta}(t)}|^p$ and stochastic terms
  $$\sup_{0\le s \le T}\left|\int_0^s \sigma(Y^{\eps,\delta}_{r-})dW_r \right|^p, 
 \sup_{0\le s \le T}\left|\int_0^s \int_{E} F(Y^{\eps,\delta}_{r-},x)\widetilde{N}(dr, dx) \right|^p, \thinspace p \in \{2,4\}.$$ Lemmas \ref{L:Linear-Sys-Sampling-Est-Levy} and  \ref{L:Sampling-Difference-Levy} deal with the deterministic quantities and Lemma \ref{L:L2-Est-Levy} handles stochastic terms mentioned above.

\begin{lemma}\label{L:Linear-Sys-Sampling-Est-Levy}
Let $y_t^\delta$ and $y_t$  be the solutions of \eqref{E:ie-sampling} and \eqref{E:closed-system}, respectively. Then, for any fixed $T > 0$, and $p \in \BN,$ there exists a positive constant $C_{\ref{L:Linear-Sys-Sampling-Est-Levy}}(T)$ 
such that 
\begin{equation*}
\begin{aligned}
\sup_{0\le t \le T}|y_t^\delta|^p \le C_{\ref{L:Linear-Sys-Sampling-Est-Levy}}(T) \qquad \text{and} \qquad
\sup_{0\le t \le T}|y_t|^p \le C_{\ref{L:Linear-Sys-Sampling-Est-Levy}}(T).
\end{aligned}
\end{equation*}
\end{lemma}
\begin{proof}[Proof of Lemma \ref{L:Linear-Sys-Sampling-Est-Levy}]
The proof follows from the integral representations of $y_t$ and $y_t^\delta$ followed by H$\ddot{\text{o}}$lder's and Gronwall's inequalities.
\end{proof}

\begin{lemma}\label{L:Sampling-Difference-Levy}
Let $y_t$ be the solution of \eqref{E:closed-system}. Then, for any fixed $T > 0,$ and $p\in \BN,$ there exists a positive constant $C_{\ref{L:Sampling-Difference-Levy}}(T)$ 
such that 
$|y_t-y_{\pi_{\delta}(t)}|^p \le \delta^p C_{\ref{L:Sampling-Difference-Levy}}(T), \thinspace t\in[0,T].$
\end{lemma}
\begin{proof}[Proof of Lemma \ref{L:Sampling-Difference-Levy}]
Recalling the integral representation of $y_t,$ we have
\begin{equation*}
|y_t-y_{\pi_{\delta}(t)}|^p = \left|\int_{\pi_{\delta}(t)}^t c(y_s,y_s)\thinspace ds \right|^p.
\end{equation*}
Using the linear growth of $c$, the fact $t-\pi_{\delta}(t)< \delta$ and Lemma \ref{L:Linear-Sys-Sampling-Est-Levy}, we obtain the required bound. 
\end{proof}

\begin{lemma}\label{L:L2-Est-Levy}
Let $Y_t^{\eps,\delta}$ be the solution of \eqref{E:sde}. Then, for any fixed $T > 0,$ and $p \in \{2,4\},$ there exists a positive constant $C_{\ref{L:L2-Est-Levy}}(T)$ 
such that
\begin{equation*}\label{E:L2-Est-Levy}
\begin{aligned}
\BE \left[\sup_{0\le s \le T}\left|Y_s^{\eps,\delta}\right|^p \right]  \le C_{\ref{L:L2-Est-Levy}}(T) ,\quad \quad 
\BE \left[\sup_{0\le s \le T}\left|\int_0^s \sigma(Y^{\eps,\delta}_{r-})\thinspace dW_r \right|^p\right] \le C_{\ref{L:L2-Est-Levy}}(T) \quad \text{and}
\end{aligned}
\end{equation*}
\begin{equation*}
\BE \left[\sup_{0\le s \le T}\left|\int_0^s \int_{E} F(Y^{\eps,\delta}_{r-},x)\widetilde{N}(dr, dx) \right|^p\right]  \le C_{\ref{L:L2-Est-Levy}}(T).
\end{equation*}
\end{lemma}
\begin{proof}[Proof of Lemma \ref{L:L2-Est-Levy}]
Let $|\cdot|$ be the one norm. We prove the lemma for $p=4$; for the sake of brevity, we omit the details of the proof for $p=2$ case. Using equations \eqref{E:Triangle-type-ineq-Levy}, \eqref{E:cadlag-property},  
\eqref{E:sde}, H$\ddot{\text{o}}$lder's inequality followed by taking supremum over $[0,T]$ and then expectation on both sides, we get
\begin{multline*}
\BE\left[\sup_{0\le s \le T}|Y_s^{\eps,\delta}|^4\right]\lesssim 1+  \int_0^T \BE \left[\sup_{0\le r \le s}|Y_r^{\eps,\delta}|^4 \right] ds + \eps^4 \BE \left[\sup_{0\le s \le T}\left|\int_0^s\sigma(Y^{\eps,\delta}_{r-})\thinspace dW_r\right|^4 \right]\\
+ \eps^4 \BE \left[\sup_{0\le s \le T}\left|\int_0^s\int_{E}F(Y^{\eps,\delta}_{r-},x)\widetilde{N}(dr,dx)
\right|^4 \right].
\end{multline*}
Using Doob's maximal inequality \cite[Theorem 1.3.8(iv)]{KS91}, \cite[Theorem 2.1.5]{applebaum2009levy}, we obtain
\begin{multline}\label{E:Exp-L2}
\BE\left[\sup_{0\le s \le T}|Y_s^{\eps,\delta}|^4\right]\lesssim 1+  \int_0^T \BE \left[\sup_{0\le r \le s}|Y_r^{\eps,\delta}|^4 \right] ds + \eps^4 \BE \left|\int_0^T\sigma(Y^{\eps,\delta}_{s-})\thinspace dW_s\right|^4\\
+ \eps^4 \BE \left[\sup_{0\le s \le T}\left|\int_0^s\int_{E}F(Y^{\eps,\delta}_{r-},x)\widetilde{N}(dr,dx)
\right|^4 \right].
\end{multline}
Since $\left|\int_0^T \sigma(Y_{s-}^{\eps,\delta})dW_s \right|^4 \lesssim  \sum_{i=1}^n\sum_{j=1}^n\left|\int_0^T \sigma_{ji}(Y_{s-}^{\eps,\delta})dW_s^i \right|^4$ for column vectors $\sigma_i\in \BR^n$, $i=1,...,n;$ using martingale moment inequalities followed by linear growth condition on $\sigma$ from Assumption \ref{A:Linear-growth-condition-Levy}, equation \eqref{E:cadlag-property}, we get
$\BE\left|\int_0^T \sigma(Y_{s-}^{\eps,\delta})dW_s \right|^4  \lesssim \BE \left( \int_0^T [1+ \sup_{ 0 \le r \le s}|Y_{r}^{\eps,\delta}|^2]ds \right)^2.$ 
Therefore, 
\begin{equation}\label{E:L2-Est-1-Levy}
\BE\left|\int_0^T \sigma(Y_{s-}^{\eps,\delta})dW_s \right|^4 \lesssim  T+ \int_0^T \BE \left[\sup_{0 \le r \le s}|Y_{r}^{\eps,\delta}|^4 \right]ds. 
\end{equation}

We now consider the term 
 $\BE \left[\sup_{0\le s \le T}\left|\int_0^s\int_{E}F(Y^{\eps,\delta}_{r-},x)\widetilde{N}(dr,dx)
\right|^4 \right]$ in equation \eqref{E:Exp-L2}.
Since \\  $\left|\int_0^T\int_{E}F(Y^{\eps,\delta}_{s-},x)\widetilde{N}(ds,dx) \right|^4 \lesssim \sum_{i=1}^{n}\left|\int_0^T\int_{E}F_i(Y^{\eps,\delta}_{s-},x)\widetilde{N}(ds,dx)\right|^4,$ employing Kunita's first inequality (moment inequality for general L{\'e}vy-type stochastic integrals)\cite[Theorem 4.4.23]{applebaum2009levy}, linear growth hypothesis on $F$, we have 
\begin{equation*}
\begin{aligned}
\BE \left[\sup_{0\le s \le T}\left|\int_0^s\int_{E}F(Y^{\eps,\delta}_{r-},x)\widetilde{N}(dr,dx)
\right|^4 \right] &\lesssim \sum_{i=1}^{n} \BE \sup_{0\le s \le T}\left|\int_0^s\int_{E}F_i(Y^{\eps,\delta}_{r-},x)\widetilde{N}(dr,dx)
\right|^4 \\
& \lesssim \sum_{i=1}^{n} \BE \left(\int_0^T\int_{E}\left|F_i(Y^{\eps,\delta}_{s-},x)\right|^2 \nu(dx) \thinspace ds \right)^2\\
&\qquad \qquad \thinspace \qquad \qquad+ \sum_{i=1}^{n} \BE \int_0^T\int_{E}\left|F_i(Y^{\eps,\delta}_{s-},x)\right|^4  \nu(dx) \thinspace ds.
\end{aligned}
\end{equation*}
Apply H$\ddot{\text{o}}$lder's inequality with $\nu(E)<\infty$ to get
\begin{equation}\label{E:L2-Est-2-Levy-general}
\begin{aligned}
\BE \left[\sup_{0\le s \le T}\left|\int_0^s\int_{E}F(Y^{\eps,\delta}_{r-},x)\widetilde{N}(dr,dx)
\right|^4 \right] &\lesssim \sum_{i=1}^{n}\BE \int_0^T\int_{E}\left|F_i(Y^{\eps,\delta}_{s-},x)\right|^4  \nu(dx) \thinspace ds \\
& \qquad \qquad \qquad \qquad  \lesssim T+ \int_0^T \BE \left[\sup_{0 \le r \le s}|Y_{r}^{\eps,\delta}|^4 \right]ds.
\end{aligned}
\end{equation}
Now, combination of estimates from equations \eqref{E:Exp-L2}, \eqref{E:L2-Est-1-Levy}, \eqref{E:L2-Est-2-Levy-general} yields the required result.
\end{proof}

We are now in the position to prove Proposition \ref{P:LLN-Stoc-App-Levy}. 
\begin{proof}[Proof of Proposition \ref{P:LLN-Stoc-App-Levy}]
For $t \ge 0, \thinspace p \in \{2,4\}$ using the integral representations of $y_t^\delta$ and $Y_t^{\eps,\delta}$ from \eqref{E:ie-sampling} and \eqref{E:sde} followed by taking norm on both sides of the above equation and then using \eqref{E:Triangle-type-ineq-Levy} and H$\ddot{\text{o}}$lder's inequality, we get\footnote{For any $t\ge 0$ and $\delta>0$, the continuity of $y_t^\delta$ implies $y_t^\delta= y_{t-}^\delta$. We will be using this fact in case of $y_t$ as well (without explicit mention).}
\begin{multline*}
|Y_t^{\eps,\delta}-y_t^\delta|^p \lesssim  t^{p-1} \int_0^t\left|c\left(Y^{\eps,\delta}_{s-}, Y^{\eps,\delta}_{\pi_{\delta}(s)-}\right) -  c\left(y^\delta_{s-}, y^\delta_{\pi_{\delta}(s)-}\right)\right|^p ds 
\\ + \eps^p \left|\int_0^t \sigma(Y^{\eps,\delta}_{s-}) \thinspace dW_s\right|^p 
+ \eps^p \left|\int_0^t \int_{E}F(Y^{\eps,\delta}_{s-},x)\widetilde{N}(ds,dx)\right|^p.
\end{multline*}
Using Lipschitz continuity of $c$, then taking supremum on right side and using \eqref{E:cadlag-property}, we have
\begin{multline*}
|Y_t^{\eps,\delta}-y_t^\delta|^p \lesssim t^{p-1} \int_0^t \sup_{0 \le s \le r}|Y^{\eps,\delta}_{s}-y^\delta_{s}|^p \thinspace dr + \eps^p \sup_{0 \le s \le t} \left|\int_0^s \sigma(Y^{\eps,\delta}_{r-})\thinspace dW_r\right|^p\\
 + \eps^p \sup_{0 \le s \le t} \left|\int_0^s \int_{E}F(Y^{\eps,\delta}_{r-},x)\widetilde{N}(dr,dx)\right|^p.
\end{multline*}
Now, as $t \in [0,T],$ and the right hand side of the above equation is non-decreasing in time $t$, taking supremum over $[0,T]$ followed by expectation on both sides, we obtain
\begin{equation*}
\begin{aligned}
\BE\left[\sup_{0 \le s \le T}|Y_s^{\eps,\delta}-y_s^\delta|^p \right] & \lesssim   \int_0^T \BE \sup_{0 \le s \le r}|Y^{\eps,\delta}_{s}-y^\delta_{s}|^p dr  + \eps^p \BE \sup_{0 \le s \le T} \left|\int_0^s \sigma(Y^{\eps,\delta}_{r-}) dW_r\right|^p \\
&  \quad \qquad \qquad  \qquad \qquad  \qquad \qquad + \eps^p \BE \left[ \sup_{0 \le s \le T} \left|\int_0^s \int_{E}F(Y^{\eps,\delta}_{r-},x)\widetilde{N}(dr,dx)\right|^p \right].
 \end{aligned}
\end{equation*}
Next, using the estimates for the terms $\BE \left[\sup_{0 \le s \le T} \left|\int_0^s \sigma(Y^{\eps,\delta}_{r-}) dW_r\right|^p \right]$ and \\ $\BE \left[ \sup_{0 \le s \le T} \left|\int_0^s \int_{E}F(Y^{\eps,\delta}_{r-},x)\widetilde{N}(dr,dx)\right|^p \right]$ from Lemma \ref{L:L2-Est-Levy} followed by Gronwall's inequality, we get the required result.
\end{proof}

\section{Analysis of fluctuations: Regimes 1 and 2}\label{S:FCLT-Levy}

In this section, we prove our second main result (Theorem \ref{T:fluctuations-R-1-2-Levy}). Our thoughts are organized as follows: in Section \ref{SS:Proof-CLT-Levy}, we state Propositions \ref{P:Main-Term-App-Levy}, \ref{P:Remainder-term-est}, \ref{P:White-Noise-Term-Est}, \ref{P:Poisson-Noise-Est} which are the main building blocks of the proof of Theorem \ref{T:fluctuations-R-1-2-Levy}. Proposition \ref{P:Main-Term-App-Levy} is a key result in the study of the fluctuation analysis of the rescaled process $Z_t^{\eps,\delta}\triangleq \eps^{-1}(Y_t^{\eps,\delta}-y_t)$ solving \eqref{E:Central-Limit-Eq-Levy} which enables us to get the extra drift term in the limiting {\sc sde} \eqref{E:lim-fluct-R-1-2-Levy}. Propositions \ref{P:Remainder-term-est}, \ref{P:White-Noise-Term-Est} and \ref{P:Poisson-Noise-Est} deal with the remainder term and the stochastic terms 
$\int_0^t \{\sigma(Y_{s-}^{\eps,\delta})-\sigma(y_{s-})\}dW_s$ and $\int_0^t\int_{E} \{F(Y_{s-}^{\eps,\delta},x)-F(y_{s-},x)\}\widetilde{N}(ds,dx)$, respectively and they are small in the limit. The proofs of Propositions \ref{P:Main-Term-App-Levy}, \ref{P:Remainder-term-est}, \ref{P:White-Noise-Term-Est}, \ref{P:Poisson-Noise-Est} are presented in Section \ref{SS:Proofs-Prop-CLT-Levy}. Proof of Proposition \ref{P:Main-Term-App-Levy} does require a couple of auxiliary lemmas; the proofs of these lemmas are deferred to Section \ref{S:Main-Term-App-Levy} to reduce clutter.




\subsection{Proof of Theorem \ref{T:fluctuations-R-1-2-Levy}}\label{SS:Proof-CLT-Levy}
For Regimes 1 and 2, we have already noticed in Section \ref{S:PS-MainResult-Levy} that the rescaled fluctuation process $Z_t^{\eps,\delta}\triangleq \eps^{-1}(Y_t^{\eps,\delta}-y_t)$ satisfies equation \eqref{E:Central-Limit-Eq-Levy}. Now, in order to show the effective (independent of small parameters $\eps$ and $\delta$) process $Z_t$, solving equation \eqref{E:lim-fluct-R-1-2-Levy}, describes the limiting behavior of process $Z_t^{\eps,\delta}$,
one needs to demonstrate the following equation holds in a suitable path-wise sense:
\begin{equation}\label{E:ell-Levy}
\lim_{\substack{\eps,\delta \searrow 0\\ \delta/\eps \to \ell}}\int_0^t D_2c(y_s,y_s)\frac{Y_{s-}^{\varepsilon, \delta}-Y_{\pi_\delta(s)-}^{\varepsilon, \delta}}{\eps}\thinspace ds  =\dd(t), \quad \text{where} \quad
\dd(t) \triangleq \frac{\ell}{2} \int_0^t D_2c(y_s,y_s)\cdot c(y_s,y_s) \thinspace ds.
\end{equation}
Equation \eqref{E:ell-Levy} above is important as it gives the identification of the drift term $\dd(t)$ in \eqref{E:lim-fluct-R-1-2-Levy}. Next, in Proposition \ref{P:Main-Term-App-Levy}, the convergence of $ \int_0^t D_2c(y_s,y_s)\frac{{Y_{s-}^{\varepsilon, \delta}-Y_{\pi_\delta(s)-}^{\varepsilon, \delta}}}{\eps}\thinspace ds$, as $\eps,\delta \searrow 0$,  to $\dd(t)$ is shown and this allows one to replace the quantity $\eps^{-1} \int_0^t D_2c(y_s,y_s)({Y_{s-}^{\varepsilon, \delta}-Y_{\pi_\delta(s)-}^{\varepsilon, \delta}})\thinspace ds$ in equation \eqref{E:Central-Limit-Eq-Levy} by $\dd(t).$

\begin{proposition}\label{P:Main-Term-App-Levy}
Let $y_t$ and $Y_t^{\varepsilon,\delta}$  solve \eqref{E:closed-system} and \eqref{E:sde}, respectively. Then, for any fixed $T>0,$ there exists a positive constant $C_{\ref{P:Main-Term-App-Levy}}(T)$ such that for any $0< \eps <\eps_0$ 
\begin{equation*}
\BE\left[ \sup_{0 \le t \le T}\left|\int_0^t D_2c(y_s,y_s)\frac{Y_{s-}^{\varepsilon, \delta}-Y_{\pi_\delta(s)-}^{\varepsilon, \delta}}{\eps}\thinspace ds - \dd(t) \right|^2 \right]\\ \le (\ell+1)^2\{\eps^2(1+\delta)+ \delta + \varkappa^2(\eps) \}C_{\ref{P:Main-Term-App-Levy}}(T),
\end{equation*}
where $\varkappa(\eps)\searrow 0$ is as in equation \eqref{E:kappa-Levy} and $\ell(t)$ is defined in \eqref{E:ell-Levy}.
\end{proposition}

Proposition \ref{P:Remainder-term-est} deals with the remainder term and shows it is small in an appropriate sense.
\begin{proposition}\label{P:Remainder-term-est}
Let ${\sf R}_t^{\eps,\delta}$ be defined as follows:
\begin{multline*}
{\sf R}_t^{\eps,\delta}\triangleq \int_0^t \left[\frac{c\left(Y^{\eps,\delta}_{s-}, Y^{\eps,\delta}_{{\pi_\delta(s)-}}\right)-c(y_s,y_s)}{\eps} - D_1c(y_s,y_s)Z^{\eps,\delta}_{s-} - D_2c(y_s,y_s)Z^{\eps,\delta}_{s-} \right. \\
\left. -D_2c(y_s,y_s)\left(\frac{Y^{\eps,\delta}_{{\pi_\delta(s)-}}-Y_{s-}^{\eps,\delta}}{\eps} \right)\right]ds. 
\end{multline*}
Then, for any fixed $T>0,$ there exists a positive constant $C_{\ref{P:Remainder-term-est}}(T)$ such that 
\begin{equation*}
\BE\left[\sup_{0\le t \le T}|{\sf R}_t^{\eps,\delta}|^2\right] \le C_{\ref{P:Remainder-term-est}}(T)(\eps^{4}+ \delta^{4}). 
\end{equation*}
\end{proposition}
Propositions \ref{P:White-Noise-Term-Est} and \ref{P:Poisson-Noise-Est} below allow one to replace the stochastic terms  $\sigma(Y_{s-}^{\eps,\delta})$ and $F(Y_{s-}^{\eps,\delta},x)$ in equation \eqref{E:Central-Limit-Eq-Levy} by the deterministic terms  $\sigma(y_{s-})$ and $F(y_{s-},x)$ respectively.
\begin{proposition}\label{P:White-Noise-Term-Est}
Let $y_t$ and $Y_t^{\varepsilon,\delta}$  solve \eqref{E:closed-system} and \eqref{E:sde}, respectively. Then, for any fixed $T>0,$ there exists a positive constant $C_{\ref{P:White-Noise-Term-Est}}(T)$ such that for any $\eps,\delta>0$
\begin{equation*}
\BE\left[ \sup_{0\le t \le T}\left|\int_0^t \{\sigma(Y_{s-}^{\eps,\delta})-\sigma(y_{s-})\}\thinspace dW_s\right|^2 \right] \le  (\eps^2 +\delta^2)C_{\ref{P:White-Noise-Term-Est}}(T).
\end{equation*}
\end{proposition}

\begin{proposition}\label{P:Poisson-Noise-Est}
Let $y_t$ and $Y_t^{\varepsilon,\delta}$  solve \eqref{E:closed-system} and \eqref{E:sde}, respectively. Then, for any fixed $T>0,$ there exists a positive constant $C_{\ref{P:Poisson-Noise-Est}}(T)$ such that for any $\eps,\delta>0$
\begin{equation*}
\BE\left[ \sup_{0\le t \le T}\left|\int_0^t\int_{E} \{F(Y_{s-}^{\eps,\delta},x)-F(y_{s-},x)\}\widetilde{N}(ds,dx)\right|^2 \right] \le  (\eps^2 +\delta^2)C_{\ref{P:Poisson-Noise-Est}}(T).
\end{equation*}
\end{proposition}


As noted earlier, the proofs of these results are given in Section \ref{SS:Proofs-Prop-CLT-Levy}. We now give the proof of Theorem \ref{T:fluctuations-R-1-2-Levy}.
\begin{proof}[Proof of Theorem \ref{T:fluctuations-R-1-2-Levy}]
Here, $|\cdot|$ represents the one norm. Recall that $Y^{\eps,\delta}_t - y_t - \eps Z_t = \eps (Z^{\eps,\delta}_t - Z_t )$, where $Z^{\eps,\delta}_t$ and $Z_t$ are given by \eqref{E:Central-Limit-Eq-Levy} and \eqref{E:lim-fluct-R-1-2-Levy}, respectively. Hence,

\begin{equation*}
\begin{aligned}
Z^{\eps,\delta}_t - Z_t &= \int_0^t \{D_1c(y_s,y_s)+D_2c(y_s,y_s)\}(Z_{s-}^{\eps,\delta}-Z_{s-})\thinspace ds + \int_0^t D_2c(y_s,y_s)\left(\frac{Y^{\eps,\delta}_{{\pi_\delta(s)-}}-Y_{s-}^{\eps,\delta}}{\eps} \right)ds - \dd(t) \\
& \qquad \qquad \quad+ \int_0^t \{\sigma(Y_{s-}^{\eps,\delta})-\sigma(y_{s-})\} \thinspace dW_s   + \int_0^t\int_{E} \{F(Y_{s-}^{\eps,\delta},x)-F(y_{s-},x)\}\widetilde{N}(ds,dx) + {\sf R}_t^{\eps,\delta},
\end{aligned}
\end{equation*}
where ${\sf R}_t^{\eps,\delta}$ is defined in equation \eqref{E:Central-Limit-Eq-Levy}. Using Assumption \ref{A:Derivative} (in particular, $|D_1c(x,y)|\le C(1+|y|)$ and $|D_2c(x,y)|\le C$), we have

\begin{equation*}
\begin{aligned}
|Z^{\eps,\delta}_t - Z_t|^2 & \lesssim t \int_0^t |Z^{\eps,\delta}_{s-} - Z_{s-}|^2 \thinspace ds +  \left|\int_0^t D_2c(y_s,y_s) \frac{Y_{s-}^{\varepsilon, \delta}-Y_{\pi_\delta(s)-}^{\varepsilon, \delta}}{\eps}\thinspace ds - \dd(t)\right|^2\\
& \quad  +\left|\int_0^t \{\sigma(Y_{s-}^{\eps,\delta})-\sigma(y_{s-})\} \thinspace dW_s\right|^2 + |{\sf R}_t^{\eps,\delta}|^2  +\left|\int_0^t\int_{E} \{F(Y_{s-}^{\eps,\delta},x)-F(y_{s-},x)\}\widetilde{N}(ds,dx)\right|^2.
\end{aligned}
\end{equation*}
Next, taking supremum over $[0,T]$ followed by the expectation  on both sides of the above equation, we have
\begin{equation*}
\begin{aligned}
\BE \left[\sup_{0\le t \le T}|Z^{\eps,\delta}_t - Z_t|^2\right] & \lesssim  T\int_0^T \BE \left[\sup_{0\le s \le T}|Z^{\eps,\delta}_{s-} - Z_{s-}|^2\right] du \\
& \qquad  + \BE \left[\sup_{0\le t \le T}\left|\int_0^t D_2c(y_s,y_s)\frac{Y_{s-}^{\varepsilon, \delta}-Y_{\pi_\delta(s)-}^{\varepsilon, \delta}}{\eps}\thinspace ds - \dd(t)\right|^2 \right]\\
& \qquad \qquad +\BE\left[\sup_{0\le t \le T}\left|\int_0^t \{\sigma(Y_{s-}^{\eps,\delta})-\sigma(y_{s-})\}\thinspace dW_s\right|^2\right]+ \BE\left[\sup_{0\le t \le T}|{\sf R}_t^{\eps,\delta}|^2\right]\\
& \qquad \qquad \qquad \qquad  +\BE \left[\sup_{0\le t \le T}\left|\int_0^t\int_{E} \{F(Y_{s-}^{\eps,\delta},x)-F(y_{s-},x)\}\widetilde{N}(ds,dx)\right|^2\right].
\end{aligned}
\end{equation*}
Using Propositions \ref{P:Main-Term-App-Levy}, \ref{P:Remainder-term-est}, \ref{P:White-Noise-Term-Est} and \ref{P:Poisson-Noise-Est} and equation \eqref{E:cadlag-property}, followed by Gronwall's inequality, we get the required bound. 
\end{proof}

\subsection{Proofs of Propositions \ref{P:Main-Term-App-Levy} through \ref{P:Poisson-Noise-Est}}\label{SS:Proofs-Prop-CLT-Levy}
In this section, we prove the main building blocks (Propositions \ref{P:Main-Term-App-Levy} through \ref{P:Poisson-Noise-Est}) of the proof of Theorem \ref{T:fluctuations-R-1-2-Levy}. A summary of this section is as follows. In order to prove Proposition \ref{P:Main-Term-App-Levy}, we write the term $\eps^{-1}\int_0^t D_2c(y_s,y_s)({Y_{s}^{\varepsilon, \delta}-Y_{\pi_\delta(s)}^{\varepsilon, \delta}})\thinspace ds$ as a sum of ${\sf M}_i^{\eps,\delta}(t), 1\le i \le 4$, in Lemma \ref{L:M-terms-levy}. Next, bounds on the terms ${\sf M}_i^{\eps,\delta}(t), 1\le i \le 4,$ are obtained in Lemmas \ref{L:M1-Est-Levy} to \ref{L:M4-Est-PoissonNoise} stated below (without proofs in this section) followed by the proof of Proposition \ref{P:Main-Term-App-Levy}. A close inspection shows that the extra drift term $\dd(t)\triangleq \frac{\ell}{2}\int_0^t D_2c(y_s,y_s)c(y_s,y_s)\thinspace ds$ in the limiting {\sc sde} \eqref{E:lim-fluct-R-1-2-Levy} is due to ${\sf M}_2^{\eps,\delta}(t),$ whereas the other ${\sf M}_i^{\eps,\delta}(t), i=1,3,4$ are small in a suitable sense. To keep the presentation clean, the proofs of these lemmas are deferred to Section \ref{S:Main-Term-App-Levy}. We close out this section with the proofs of Propositions \ref{P:Remainder-term-est}, \ref{P:White-Noise-Term-Est} and \ref{P:Poisson-Noise-Est}. 

Moving in the direction of the proof of Proposition \ref{P:Main-Term-App-Levy}, equation \eqref{E:Int-Equal-Levy} below enables us to replace the integral $\int_0^t D_2c(y_s,y_s)(Y_{s-}^{\varepsilon, \delta}-Y_{\pi_\delta(s)-}^{\varepsilon, \delta})\thinspace ds$ by $\int_0^t D_2c(y_s,y_s) (Y_{s}^{\varepsilon, \delta}-Y_{\pi_\delta(s)}^{\varepsilon, \delta})\thinspace ds$ in equation \eqref{E:ell-Levy}. Hence, we will be getting estimate for the term $\BE[ \sup_{0 \le t \le T}|\eps^{-1}\int_0^t D_2c(y_s,y_s)({Y_{s}^{\varepsilon, \delta}-Y_{\pi_\delta(s)}^{\varepsilon, \delta}})\thinspace ds - \dd(t)|^2 ]$ rather than 
$\BE[ \sup_{0 \le t \le T}|\eps^{-1}\int_0^t D_2c(y_s,y_s)({Y_{s-}^{\varepsilon, \delta}-Y_{\pi_\delta(s)-}^{\varepsilon, \delta}})\thinspace ds - \dd(t)|^2 ]$ in Proposition \ref{P:Main-Term-App-Levy}.

\begin{lemma}
Let $Y_t^{\eps,\delta}$ be the solution of equation \eqref{E:sde}. Then, for any $t\in [0,T],$ we have
\begin{equation}\label{E:Int-Equal-Levy}
\int_0^t Y_s^{\eps,\delta}\thinspace ds = \int_0^t Y_{s-}^{\eps,\delta}\thinspace ds.
\end{equation}
\end{lemma}
\begin{proof}
We note that the stochastic process $Y_t^{\eps,\delta}$ can have at most countable number of discontinuities, as it is a c\`{a}dl\`{a}g process. Therefore, the jump process $\Delta Y_t^{\eps,\delta}\triangleq Y_t^{\eps,\delta}-Y_{t-}^{\eps,\delta}=0$ for almost every $t\in [0,T],$ and hence $\int_0^t \Delta Y_s^{\eps,\delta} \thinspace ds=0.$ The required result now follows.
\end{proof}

\begin{lemma}\label{L:M-terms-levy}
Let $Y_t^{\eps,\delta}$ be the solution of equation \eqref{E:sde} and $\xi(y_s)\triangleq D_2(y_s,y_s)$. Then, for any fixed $T>0$, $t\in[0,T],$ and $\eps,\delta>0,$ we have 
\begin{equation}\label{E:M-Rep-Levy}
\int_0^t \xi(y_s) \frac{Y_{s}^{\varepsilon, \delta}-Y_{\pi_\delta(s)}^{\varepsilon, \delta}}{\eps}\thinspace ds = \sum_{i=1}^{4}{\sf M}_i^{\eps,\delta}(t),\quad \text{where}
\end{equation} 
\begin{equation*}
\begin{aligned}
{\sf M}_1^{\eps,\delta} & \triangleq   \int_0^t \xi(y_s) \int_{\pi_\delta(s)}^s \frac{c(Y_{r-}^{\eps,\delta}, Y_{\pi_\delta(r)-}^{\varepsilon, \delta})-c(Y_{\pi_\delta(r)-}^{\varepsilon, \delta}, Y_{\pi_\delta(r)-}^{\varepsilon, \delta})}{\eps} \thinspace dr \thinspace ds\\
&\qquad \quad+ \int_0^t \xi(y_s) \int_{\pi_\delta(s)}^s \frac{c(Y_{\pi_\delta(r)-}^{\varepsilon, \delta}, Y_{\pi_\delta(r)-}^{\varepsilon, \delta})}{\eps} \thinspace dr \thinspace ds  - \int_0^t \xi(y_s) \int_{\pi_\delta(s)}^s \frac{c(y_{\pi_\delta(r)}, y_{\pi_\delta(r)})}{\eps} \thinspace dr \thinspace ds,\\
{\sf M}_2^{\eps,\delta}+ {\sf M}_3^{\eps,\delta}  & \triangleq \int_0^t \xi(y_s) \int_{\pi_\delta(s)}^s \frac{c(y_{\pi_\delta(r)}, y_{\pi_\delta(r)})}{\eps}  dr  \thinspace ds+ \int_0^t \xi(y_s) \int_{\pi_\delta(s)}^s \sigma(Y_{r-}^{\eps,\delta})\thinspace dW_r \thinspace ds,   \\
{\sf M}_4^{\eps,\delta} & \triangleq \int_0^t \xi(y_s) \int_{\pi_\delta(s)}^s \int_{E}F(Y_{r-}^{\eps,\delta},x)\widetilde{N}(dr,dx)\thinspace ds.
\end{aligned}
\end{equation*}
\end{lemma}

Now, in our next step, we show that the terms $\BE[\sup_{0 \le t \le T}|{\sf M}_1^{\eps,\delta}(t)|^2], \sup_{0 \le t \le T}|{\sf M}_2^{\eps,\delta}(t)-\dd(t)|^2, \\
 \BE[\sup_{0 \le t \le T}|{\sf M}_3^{\eps,\delta}(t)|^2]$ and $\BE[\sup_{0 \le t \le T}|{\sf M}_4^{\eps,\delta}(t)|^2]$ are small. This is accomplished in Lemmas \ref{L:M1-Est-Levy} through \ref{L:M4-Est-PoissonNoise}.

\begin{lemma}\label{L:M1-Est-Levy}
Let ${\sf M}_1^{\eps,\delta}(t)$ be defined as in equation \eqref{E:M-Rep-Levy}. Then, for any fixed $T>0,$ there exists a positive constant $C_{\ref{L:M1-Est-Levy}}(T)$ such that for any $ \eps \in (0,\eps_0),$  we have
\begin{equation*}
\BE\left[\sup_{0 \le t \le T}\left|{\sf M}_1^{\eps,\delta}(t)\right|^2\right]  \le (\ell+1)^2(\delta^2+ \eps^2)C_{\ref{L:M1-Est-Levy}}(T).
\end{equation*}
\end{lemma}

Next, we decompose ${\sf M}_2^{\eps,\delta}(t)$ defined in \eqref{E:M-Rep-Levy} as follows:

\begin{equation}\label{E:M2-decomposition}
\begin{aligned}
{\sf M}_2^{\eps,\delta}(t) &= {M}_1^{\eps,\delta}(t)+ {M}_2^{\eps,\delta}(t), \quad \text{where}\\
{M}_1^{\eps,\delta}(t) &\triangleq \int_0^t\{D_2c(y_s,y_s)-D_2c(y_{\pi_\delta(s)}, y_{\pi_\delta(s)})\}\int_{\pi_\delta(s)}^s \frac{c(y_{\pi_\delta(r)}, y_{\pi_\delta(r)})}{\eps} \thinspace dr \thinspace ds, \\
{M}_2^{\eps,\delta}(t)& \triangleq \int_0^t D_2c(y_{\pi_\delta(r)}, y_{\pi_\delta(r)})\int_{\pi_\delta(s)}^s \frac{c(y_{\pi_\delta(s)}, y_{\pi_\delta(s)})}{\eps} \thinspace dr \thinspace ds.
\end{aligned}
\end{equation}

\begin{lemma}\label{L:M1-small}
Let ${M}_1^{\eps,\delta}(t)$ be defined as in equation \eqref{E:M2-decomposition}. Then, 
\begin{equation*}
\BE\left[\sup_{0 \le t \le T}\left|{ M}_1^{\eps,\delta}(t)\right|^2\right]\le (\ell+1)^2\delta^2 C_{\ref{L:M1-small}}(T).
\end{equation*}
\end{lemma}

\begin{lemma}\label{L:M2-Est-Levy}
Let ${M}_2^{\eps,\delta}(t)$ be defined as in equation \eqref{E:M2-decomposition} with $\dd(t)$ as given in \eqref{E:ell-Levy}. Then, for any $T>0,$ there exists a positive constant $C_{\ref{L:M2-Est-Levy}}(T)$ such that for any $0<\eps< \eps_0$, we have 
\begin{equation*}
\sup_{0 \le t \le T}\left|{ M}_2^{\eps,\delta}(t)-\dd(t)\right|^2 \le [\delta^2 (\ell+1)^2+ \varkappa^2(\eps)]C_{\ref{L:M2-Est-Levy}}(T),
\end{equation*}
where $\varkappa(\eps)\searrow 0$ is as in equation \eqref{E:kappa-Levy}.
\end{lemma}

\begin{lemma}\label{L:M3-Est-WhiteNoise-Levy}
Let ${\sf M}_3^{\eps,\delta}(t)$ be defined as in equation \eqref{E:M-Rep-Levy}. Then, for any fixed $T>0,$ there exists a positive constant $C_{\ref{L:M3-Est-WhiteNoise-Levy}}(T)$ such that for any $\eps,\delta>0$, we have 
\begin{equation*}
\BE\left[\sup_{0 \le t \le T}\left|{\sf M}_3^{\eps,\delta}(t)\right|^2\right]
 \le \delta C_{\ref{L:M3-Est-WhiteNoise-Levy}}(T).
\end{equation*}
\end{lemma}

\begin{lemma}\label{L:M4-Est-PoissonNoise}
Let ${\sf M}_4^{\eps,\delta}(t)$ be defined as in equation \eqref{E:M-Rep-Levy}. Then, for any fixed $T>0,$ there exists a positive constant $C_{\ref{L:M4-Est-PoissonNoise}}(T)$ such that for any $\eps,\delta>0$, we have 
\begin{equation*}
\BE\left[\sup_{0 \le t \le T}\left|{\sf M}_4^{\eps,\delta}(t)\right|^2\right]
\le \delta  C_{\ref{L:M4-Est-PoissonNoise}}(T).
\end{equation*}
\end{lemma}
The proofs of these Lemmas are given in Section \ref{S:Main-Term-App-Levy}. We now prove Proposition \ref{P:Main-Term-App-Levy}.

\begin{proof}[Proof of Proposition \ref{P:Main-Term-App-Levy}]
From Lemma \ref{L:M-terms-levy}, we obtain 
$$\int_0^t D_2c(y_s,y_s)\frac{Y_{s}^{\varepsilon, \delta}-Y_{\pi_\delta(s)}^{\varepsilon, \delta}}{\eps}\thinspace ds - \dd(t)= \sum_{i=1}^{4}{\sf M}_i^{\eps,\delta}(t)-\dd(t).$$
Therefore, from \eqref{E:M2-decomposition}, we have
\begin{multline*}
\left|\int_0^t D_2c(y_s,y_s) \frac{Y_{s}^{\varepsilon, \delta}-Y_{\pi_\delta(s)}^{\varepsilon, \delta}}{\eps}\thinspace ds - \dd(t)\right|^2 \lesssim \left|{\sf M}_1^{\eps,\delta}(t) \right|^2 +\left|{M}_1^{\eps,\delta}(t) \right|^2+\left|{M}_2^{\eps,\delta}(t)-\dd(t) \right|^2 \\ +\left|{\sf M}_3^{\eps,\delta}(t)\right|^2+\left|{\sf M}_4^{\eps,\delta}(t) \right|^2.
\end{multline*}
Now, taking supremum over $[0,T]$ on both sides followed by expectation, we have 
\begin{equation*}
\begin{aligned}
\BE & \left[ \sup_{0 \le t \le T}\left|\int_0^t D_2c(y_s,y_s) \frac{Y_{s-}^{\varepsilon, \delta}-Y_{\pi_\delta(s)-}^{\varepsilon, \delta}}{\eps}\thinspace ds - \dd(t) \right|^2 \right]\lesssim 
\BE\left[\sup_{0 \le t \le T}\left|{\sf M}_1^{\eps,\delta}(t)\right|^2\right] + \sup_{0 \le t \le T}\left|{M}_1^{\eps,\delta}(t)\right|^2  \\
& \qquad \qquad \qquad \qquad \qquad \quad \quad + \sup_{0 \le t \le T}\left|{M}_2^{\eps,\delta}(t)-\dd(t)\right|^2 + \BE\left[\sup_{0 \le t \le T}\left|{\sf M}_3^{\eps,\delta}(t)\right|^2\right] + \BE\left[\sup_{0 \le t \le T}\left|{\sf M}_4^{\eps,\delta}(t)\right|^2\right].
\end{aligned}
\end{equation*}
The required estimate is obtained by applying Lemmas \ref{L:M1-Est-Levy} through \ref{L:M4-Est-PoissonNoise} to handle the right side of the above equation.
\end{proof}

\begin{proof}[Proof of Proposition \ref{P:Remainder-term-est}]
For any $t\in [0,T],$ setting 
\begin{multline*}
S_t^{\eps,\delta}\triangleq \left[\frac{c\left(Y^{\eps,\delta}_{t-}, Y^{\eps,\delta}_{{\pi_\delta(t)-}}\right)-c(y_t,y_t)}{\eps} - D_1c(y_t,y_t)Z^{\eps,\delta}_{t-} - D_2c(y_t,y_t)Z^{\eps,\delta}_{t-} \right. \\
\left. -D_2c(y_t,y_t)\left(\frac{Y^{\eps,\delta}_{{\pi_\delta(t)-}}-Y_{t-}^{\eps,\delta}}{\eps} \right)\right],
\end{multline*}
we get ${\sf R}_t^{\eps,\delta}=\int_0^t S_s^{\eps,\delta}\thinspace ds.$ Next, it is convenient to write $S_t^{\eps,\delta}=\left(S_t^{\eps,\delta}(1),\cdots, S_t^{\eps,\delta}(n)\right)^\top$, where  
\begin{multline*}
S_t^{\eps,\delta}(i) \triangleq  \frac{c_i\left(Y^{\eps,\delta}_{t-}, Y^{\eps,\delta}_{{\pi_\delta(t)-}}\right)-c_i(y_t,y_t)}{\eps} - \{\nabla_x c_i(y_t,y_t) + \nabla_y c_i(y_t,y_t)\}Z^{\eps,\delta}_{t-} \\
-\nabla_y c_i(y_t,y_t)\left(\frac{Y^{\eps,\delta}_{{\pi_\delta(t)-}}-Y_{t-}^{\eps,\delta}}{\eps} \right),
\end{multline*}
for each $i=1, \cdots, n$, and $\nabla_xc_i,~\nabla_yc_i$ represent the $i^{\text{th}}$ rows of the Jacobian matrices $D_1c$ and $D_2c$, respectively. Now, using Taylor's formula for each $S^{\eps,\delta}_t(i)$, we get
\begin{equation*}
S_t^{\eps,\delta}(i)=\frac{1}{2\eps}\left\langle D^2c_i(z_i(t))\left((Y^{\eps,\delta}_{t-}, Y^{\eps,\delta}_{{\pi_\delta(t)-}})-(y_t,y_t) \right), \left((Y^{\eps,\delta}_{t-}, Y^{\eps,\delta}_{{\pi_\delta(t)-}})-(y_t,y_t) \right) \right\rangle,
\end{equation*}
where $z_i(t) \triangleq  (z_i^1(t),z_i^2(t)) \in \BR^n \times \BR^n, 1\le i \le n$ are points lying on the line segment joining the points $(y_t,y_t)$ and $(Y^{\eps,\delta}_{t-}, Y^{\eps,\delta}_{{\pi_\delta(t)-}}).$ Let $|\cdot|$ be the one norm. Hence, $|S_t^{\eps,\delta}|^2 \lesssim \sum_{i=1}^{n}|S_t^{\eps,\delta}(i)|^2,$ using Cauchy-Schwarz inequality and submultiplicative property of matrix norms, we get
\begin{equation*}
\begin{aligned}
|S_t^{\eps,\delta}|^2 &\le  \sum_{i=1}^n \frac{\|D^2c_i(z_i(t))\|^2}{2^2\eps^2} \|(Y^{\eps,\delta}_{t-}, Y^{\eps,\delta}_{{\pi_\delta(t)-}})-(y_t,y_t)\|^2 \|(Y^{\eps,\delta}_{t-}, Y^{\eps,\delta}_{{\pi_\delta(t)-}})-(y_t,y_t)\|^2\\
&\le C(1/2^2\eps^2)\sum_{i=1}^n |D^2c_i(z_i(t))|^2 |Y^{\eps,\delta}_{t-}-y_t, Y^{\eps,\delta}_{{\pi_\delta(t)-}}-y_t|^{4}\\
\end{aligned}
\end{equation*}
\begin{equation*}
\begin{aligned}
& \lesssim (1/2^2\eps^2)\sum_{i=1}^n |D^2c_i(z_i(t))|^2 \{|Y^{\eps,\delta}_{t-}-y_t|^{4}+ |Y^{\eps,\delta}_{{\pi_\delta(t)-}}-y_t|^{4} \}.
\end{aligned}
\end{equation*}
Hence, by H$\ddot{\text{o}}$lder's inequality and employing a simple algebra, we have
\begin{multline*}
\BE \left[\sup_{0 \le t \le T }|S_t^{\eps,\delta}|^2\right] \lesssim \sum_{i=1}^n \left[  \frac{\left(\BE \sup_{0 \le t \le T }|D^2c_i(z_i(t))|^{4}\right)^\frac{1}{2}}{2^2\eps^2} \times \right. \\
\left. \left(\BE \left[ \sup_{0\le t \le T}|Y^{\eps,\delta}_{t}-y_t|^{8} + \sup_{0\le t \le T} |y_{{\pi_\delta(t)-}}-y_t|^{8} \right]\right)^\frac{1}{2}\right].
\end{multline*}
For the entries of Hessian matrix (that is denoted by $D^2c_i$) in the above equation, one require the following inequalities (Assumption \ref{A:Derivative}): for $i,j,k \in \{1,\cdots,n\}$
\begin{equation*}
\left|\frac{\partial^2 c_i (z_i(t))}{\partial z_i^1(k)\partial z_i^2(j)}\right| \le C, \quad \left|\frac{\partial^2 c_i(z_i(t))}{\partial z_i^1(k)\partial z_i^1(j)}\right| \le C (1+ |z_i^2(t)|), \quad \left|\frac{\partial^2 c_i(z_i(t))}{\partial z_i^2(k)\partial z_i^2(j)}\right| \le C.
\end{equation*}
The required estimate can now be obtained by using Lemma \ref{L:L2-Est-Levy}, \ref{L:Linear-Sys-Sampling-Est-Levy} and  Theorem \ref{T:LLN}.
\end{proof}

\begin{proof}[Proof of Proposition \ref{P:White-Noise-Term-Est}]
Let $|\cdot|$ and $\sigma_i\in \BR^n$, $1 \le i \le n,$ denote the one norm and columns of the matrix $\sigma$ respectively. 
Here,
\begin{multline*}
\begin{aligned}
\left|\int_0^t \{\sigma(Y_{s-}^{\eps,\delta})-\sigma(y_{s-})\}\thinspace dW_s\right|^2 &= \left|\sum_{i=1}^{n}\int_0^t \{\sigma_i(Y_{s-}^{\eps,\delta})-\sigma_i(y_{s-})\}\thinspace dW_s^i\right|^2\\
& \lesssim \sum_{i=1}^{n}\sum_{j=1}^{n}  \left|\int_0^t \{\sigma_{ji}(Y_{s-}^{\eps,\delta})-\sigma_{ji}(y_{s-})\}\thinspace dW_s^i\right|^2.
\end{aligned}
\end{multline*}
Now, first taking supremum over $[0,T]$ and then expectation on both sides, we get
\begin{equation*}
\BE\left[\sup_{0 \le t \le T}\left|\int_0^t \{\sigma(Y_{s-}^{\eps,\delta})-\sigma(y_{s-})\}\thinspace dW_s\right|^2\right] \lesssim \sum_{i=1}^{n}\sum_{j=1}^{n} \BE \left[ \sup_{0 \le t \le T} \left|\int_0^t \{\sigma_{ji}(Y_{s-}^{\eps,\delta})-\sigma_{ji}(y_{s-})\}\thinspace dW_s^i\right|^2\right].
\end{equation*}
Using Doob's maximal inequality followed by It{\^o} isometry and then Lipschitz continuity of $\sigma$ from Assumption \ref{A:Lip-continuity-Levy}, we get
\begin{equation*}
\begin{aligned}
\BE\left[\sup_{0 \le t \le T}\left|\int_0^t \{\sigma(Y_{s-}^{\eps,\delta})-\sigma(y_{s-})\}\thinspace dW_s\right|^2\right] & \lesssim
\sum_{i,j=1}^{n} \BE  \left|\int_0^T \{\sigma_{ji}(Y_{s-}^{\eps,\delta})-\sigma_{ji}(y_{s-})\}\thinspace dW_s^i\right|^2\\
&= \sum_{i,j=1}^{n} \BE \left[\int_0^T|\sigma_{ji}(Y_{s-}^{\eps,\delta})-\sigma_{ji}(y_{s-})|^2 \thinspace ds \right]\\
& \lesssim  \thinspace \BE \left[ \sup_{0 \le t \le T }|Y_{t}^{\eps,\delta}-y_t|^2\right],
\end{aligned}
\end{equation*}
where the last inequality uses equation \eqref{E:cadlag-property} and H$\ddot{\text{o}}$lder's inequality. Next, using Theorem \ref{T:LLN}, we get the required estimate.
\end{proof}

\begin{proof}[Proof of Proposition \ref{P:Poisson-Noise-Est}]
Let $|\cdot|$ be the one norm. Then,
\begin{equation*}
\left|\int_0^t\int_{E}\{F(Y_{s-}^{\eps,\delta},x)-F(y_{s-},x)\}\widetilde{N}(ds,dx)\right|^2 \\
 \lesssim \sum_{i=1}^{n}
\left|\int_0^t\int_{E}\{F_i(Y_{s-}^{\eps,\delta},x)-F_i(y_{s-},x)\}\widetilde{N}(ds,dx)\right|^2.
\end{equation*}
Now, first taking supremum over $[0,T]$ and then expectation on both sides, we get
\begin{multline}\label{E:Poisson-Noise-Est-1}
\BE \left[ \sup_{0\le t \le T}\left|\int_0^t\int_{E}\{F(Y_{s-}^{\eps,\delta},x)-F(y_{s-},x)\}\widetilde{N}(ds,dx)\right|^2 \right] \lesssim \\
\sum_{i=1}^{n}
\BE \left[ \sup_{0\le t \le T} \left|\int_0^t\int_{E}\{F_i(Y_{s-}^{\eps,\delta},x)-F_i(y_{s-},x)\}\widetilde{N}(ds,dx)\right|^2\right].
\end{multline}
For each $i=1,...,n$, use of Doob's maximal inequality followed by It\^{o} isometry\footnote{For a certain class of functions $f$ and the integral $\int_0^T\int_E f(t,x)\widetilde{N}(dt,dx)$ (see, \cite[Section 4.2]{applebaum2009levy} for the definition of the integral), the identity
$\BE\left|\int_0^T\int_E f(t,x)\widetilde{N}(dt,dx)\right|^2= \BE\int_0^T \int_E|f(t,x)|^2 \nu(dx) \thinspace dt$
 is called It\^{o} isometry.} yields
\begin{multline}\label{E:Poisson-Noise-Est-2}
\BE \left[ \sup_{0\le t \le T} \left|\int_0^t\int_{E}\{F_i(Y_{s-}^{\eps,\delta},x)-F_i(y_{s-},x)\}\widetilde{N}(ds,dx)\right|^2\right] \\
\qquad \qquad \qquad \qquad \qquad \qquad 
 \qquad = \BE  \left[\int_0^T\int_{E}\left|F_i(Y_{s-}^{\eps,\delta},x)-F_i(y_{s-},x)\right|^2 \nu(dx) \thinspace ds \right].
\end{multline}
Now, using Lipschitz continuity of $F$ from Assumption \ref{A:Lip-continuity-Levy} and from equations \eqref{E:Poisson-Noise-Est-1}, \eqref{E:Poisson-Noise-Est-2}, we obtain 
\begin{equation*}
\BE \left[ \sup_{0\le t \le T}\left|\int_0^t\int_{E}\{F(Y_{s-}^{\eps,\delta},x)-F(y_{s-},x)\}\widetilde{N}(ds,dx)\right|^2 \right]
 \lesssim \BE \int_0^T |Y_{s-}^{\eps,\delta}-y_{s-}|^2ds.
\end{equation*}
Finally, equation \eqref{E:cadlag-property} and Theorem \ref{T:LLN} yield the required estimate.
\end{proof}

\section{Proofs of Lemmas  \ref{L:M-terms-levy} through \ref{L:M4-Est-PoissonNoise}}\label{S:Main-Term-App-Levy}


\begin{proof}[Proof of Lemma \ref{L:M-terms-levy}]
From the integral representation of {\sc sde} \eqref{E:sde} and a simple algebra, we have
\begin{equation*}
\begin{aligned}
\int_0^t D_2c(y_s,y_s)\frac{Y_{s}^{\varepsilon, \delta}-Y_{\pi_\delta(s)}^{\varepsilon, \delta}}{\eps}\thinspace ds & =\int_0^t D_2c(y_s,y_s) \int_{\pi_\delta(s)}^s \frac{c(Y_{r-}^{\eps,\delta}, Y_{\pi_\delta(r)-}^{\varepsilon, \delta})}{\eps} \thinspace dr \thinspace ds\\
& \qquad \qquad  + \int_0^t D_2c(y_s,y_s) \int_{\pi_\delta(s)}^s \sigma(Y_{r-}^{\eps,\delta})\thinspace dW_r \thinspace ds\\
& \qquad \qquad \qquad \qquad \quad + \int_0^t D_2c(y_s,y_s) \int_{\pi_\delta(s)}^s \int_{E}F(Y_{r-}^{\eps,\delta},x)\widetilde{N}(dr,dx) \thinspace ds.
\end{aligned}
\end{equation*}
Writing $c(Y_{r-}^{\eps,\delta}, Y_{\pi_\delta(r)-}^{\varepsilon, \delta})$ as $[c(Y_{r-}^{\eps,\delta}, Y_{\pi_\delta(r)-}^{\varepsilon, \delta})-c(Y_{\pi_\delta(r)-}^{\varepsilon, \delta}, Y_{\pi_\delta(r)-}^{\varepsilon, \delta})+c(Y_{\pi_\delta(r)-}^{\varepsilon, \delta}, Y_{\pi_\delta(r)-}^{\varepsilon, \delta})+c(y_{\pi_\delta(r)}, y_{\pi_\delta(r)})-c(y_{\pi_\delta(r)}, y_{\pi_\delta(r)})]$ in the second term of the right hand side of the above equation, we get
\begin{equation*}
\begin{aligned}
(1/\eps)\int_0^t  D_2c(y_s,y_s)  [{Y_{s}^{\varepsilon, \delta}-Y_{\pi_\delta(s)}^{\varepsilon, \delta}}]\thinspace ds &= \int_0^t D_2c(y_s,y_s) \int_{\pi_\delta(s)}^s \sigma(Y_{r-}^{\eps,\delta})\thinspace dW_r \thinspace ds \\
 & + \int_0^t D_2c(y_s,y_s) \int_{\pi_\delta(s)}^s \frac{c(Y_{r-}^{\eps,\delta}, Y_{\pi_\delta(r)-}^{\varepsilon, \delta})-c(Y_{\pi_\delta(r)-}^{\varepsilon, \delta}, Y_{\pi_\delta(r)-}^{\varepsilon, \delta})}{\eps} \thinspace dr \thinspace ds\\
&  + \int_0^t D_2c(y_s,y_s) \int_{\pi_\delta(s)}^s \frac{c(Y_{\pi_\delta(r)-}^{\varepsilon, \delta}, Y_{\pi_\delta(r)-}^{\varepsilon, \delta})}{\eps} \thinspace dr \thinspace ds \\
& - \int_0^t D_2c(y_s,y_s) \int_{\pi_\delta(s)}^s \frac{c(y_{\pi_\delta(r)}, y_{\pi_\delta(r)})}{\eps} \thinspace dr \thinspace ds\\
 & + \int_0^t D_2c(y_s,y_s) \int_{\pi_\delta(s)}^s \frac{c(y_{\pi_\delta(r)}, y_{\pi_\delta(r)})}{\eps} \thinspace dr \thinspace ds\\
  & +\int_0^t D_2c(y_s,y_s) \int_{\pi_\delta(s)}^s \int_{E}F(Y_{r-}^{\eps,\delta},x)\widetilde{N}(dr,dx) \thinspace ds.
 \end{aligned}
 \end{equation*}
The right hand side of the above equation is easily recognised as the sum of ${\sf M}_i^{\eps,\delta}(t), 1 \le i \le 4.$
\end{proof}

\begin{proof}[Proof of Lemma \ref{L:M1-Est-Levy}]
Recalling the definition of ${\sf M}_1^{\eps,\delta}(t)$ from \eqref{E:M-Rep-Levy}, we have

\begin{equation*}
\begin{aligned}
{\sf M}_1^{\eps,\delta}(t)& =  \int_0^t D_2c(y_s,y_s) \int_{\pi_\delta(s)}^s \frac{c(Y_{r-}^{\eps,\delta}, Y_{\pi_\delta(r)-}^{\varepsilon, \delta})-c(Y_{\pi_\delta(r)-}^{\varepsilon, \delta}, Y_{\pi_\delta(r)-}^{\varepsilon, \delta})}{\eps} \thinspace dr \thinspace ds\\
& \quad  + \int_0^t D_2c(y_s,y_s) \int_{\pi_\delta(s)}^s \frac{c(Y_{\pi_\delta(r)-}^{\varepsilon, \delta}, Y_{\pi_\delta(r)-}^{\varepsilon, \delta})}{\eps} \thinspace dr \thinspace ds   - \int_0^t D_2c(y_s,y_s) \int_{\pi_\delta(s)}^s \frac{c(y_{\pi_\delta(r)}, y_{\pi_\delta(r)})}{\eps} \thinspace dr \thinspace ds\\
& = \int_0^t D_2c(y_s,y_s) \int_{\pi_\delta(s)}^s \frac{c(Y_{r-}^{\eps,\delta}, Y_{\pi_\delta(r)-}^{\varepsilon, \delta})-c(Y_{\pi_\delta(r)-}^{\varepsilon, \delta}, Y_{\pi_\delta(r)-}^{\varepsilon, \delta})}{\eps} \thinspace dr \thinspace ds\\
& \qquad \qquad   + \int_0^t D_2c(y_s,y_s) \int_{\pi_\delta(s)}^s \frac{c(Y_{\pi_\delta(r)-}^{\varepsilon, \delta}, Y_{\pi_\delta(r)-}^{\varepsilon, \delta})-c(y_{\pi_\delta(r)}, y_{\pi_\delta(r)})}{\eps} \thinspace dr \thinspace ds \triangleq Q_1^{\eps,\delta}(t) + Q_2^{\eps,\delta}(t).
\end{aligned}
\end{equation*}

Now, for $Q_1^{\eps,\delta}(t),$ using Assumption \ref{A:Lip-continuity-Levy}, H$\ddot{\text{o}}$lder's inequality, boundedness of $D_2c(y_s,y_s)$ and noting $(s-\pi_{\delta}(s)) \le \delta$, we have
\begin{equation*}
\begin{aligned}
|Q_1^{\eps,\delta}(t)|^2 
& \lesssim \frac{\delta}{\eps^2}\int_0^t\int_{\pi_\delta(s)}^{s}\sup_{0\le r \le s}|Y_r^{\eps,\delta}-y_r+y_r-y_{\pi_\delta(r)}+y_{\pi_\delta(r)}-Y_{\pi_\delta(r)}^{\eps,\delta}|^2\thinspace du \thinspace ds\\
& \lesssim \frac{\delta}{\eps^2}\int_0^t\int_{\pi_\delta(s)}^{s}\sup_{0\le r \le s}|Y_r^{\eps,\delta}-y_r|^2\thinspace du \thinspace ds + \frac{\delta}{\eps^2}\int_0^t\int_{\pi_\delta(s)}^{s}\sup_{0\le r \le s}|y_r-y_{\pi_\delta(r)}|^2\thinspace du \thinspace ds.
\end{aligned}
\end{equation*}
Taking supremum on left side and then taking expectation on both sides, we get
\begin{multline*}
\BE\left[\sup_{0\le t \le T}|Q_1^{\eps,\delta}(t)|^2\right] \lesssim \frac{\delta}{\eps^2}\left\{\int_0^T \int_{\pi_\delta(s)}^{s} \BE\left[ \sup_{0\le r \le s}|Y_r^{\eps,\delta}-y_r|^2\right]\thinspace du \thinspace ds \right. \\
\left.+ \int_0^T \int_{\pi_\delta(s)}^{s}\sup_{0\le r \le s}|y_r-y_{\pi_\delta(r)}|^2 \thinspace du \thinspace ds \right\}.
\end{multline*}
Theorem \ref{T:LLN} and Lemma \ref{L:Sampling-Difference-Levy} now yield $\BE\left[\sup_{0\le t \le T}|Q_1^{\eps,\delta}(t)|^2\right] \lesssim \frac{\delta^2}{\eps^2}(\delta^2+ \eps^2)e^{CT}.$ For $Q_2^{\eps,\delta}(t),$ by similar calculations to those above, we obtain $\BE\left[\sup_{0\le t \le T}|Q_2^{\eps,\delta}(t)|^2\right] \lesssim \frac{\delta^2}{\eps^2}(\delta^2+ \eps^2)e^{CT}.$ Putting these two estimates together and recalling equation \eqref{E:eps0}, we get the desired result.
\end{proof}

Before we prove Lemma \ref{L:M1-small} and \ref{L:M2-Est-Levy}, an estimate for the term 
$${R}_t^{\eps,\delta}\triangleq \frac{1}{2^2}\frac{\delta^2}{\eps^2}\left|\sum_{i=0}^{\lfloor \frac{t}{\delta}\rfloor-1}\delta D_2c(y_{i\delta},y_{i\delta})\cdot c(y_{i\delta},y_{i\delta})-\int_0^t D_2 c(y_s,y_s)\cdot c(y_s,y_s) \thinspace ds\right|^2$$
is provided in Lemma \ref{L:Remainder-Est-Levy} which shows that ${R}_t^{\eps,\delta}= \mathscr{O}(\delta^2),$ whenever $0< \eps < \eps_0.$ We will require the estimate for ${R}_t^{\eps,\delta}$ in the proof of Lemma \ref{L:M2-Est-Levy} below.

\begin{lemma}\label{L:Remainder-Est-Levy}
Let $y_t$ be the solution of equation \eqref{E:closed-system}. Then, for any fixed $T>0,$ there exists a positive constant $C_{\ref{L:Remainder-Est-Levy}}(T)$ such that 
\begin{equation*}
\begin{aligned}
\sup_{0 \le t \le T}{ R}_t^{\eps,\delta} \le \delta^2 (\ell+1)^2 C_{\ref{L:Remainder-Est-Levy}}(T).
 \end{aligned}
\end{equation*}
\end{lemma}
\begin{proof}
Let $\hh: \BR^n \to \BR^n$ be defined by $\hh(y)\triangleq D_2c(y,y)\cdot c(y,y)$. For $R_t^{\eps,\delta}$,
\begin{equation}\label{E:Rem-Eq}
\begin{aligned}
\left| \sum_{i=0}^{\lfloor \frac{t}{\delta}\rfloor-1}\delta \thinspace {\sf h}(y_{i\delta}) -\int_0^t {\sf h}(y_s)\thinspace ds\right|^2 &\lesssim 
\left|\sum_{i=0}^{\lfloor \frac{t}{\delta}\rfloor-1}\int_{i\delta}^{(i+1)\delta}\left\{{\sf h}(y_s)-{\sf h}(y_{i\delta})\right\}ds\right|^2 + \left| \int_{\delta\lfloor \frac{t}{\delta}\rfloor}^{t}{\sf h}(y_s)ds\right|^2\\
& \lesssim \sum_{i=0}^{\lfloor \frac{t}{\delta}\rfloor-1}\int_{i\delta}^{(i+1)\delta}\left|{\sf h}(y_s)-{\sf h}(y_{i\delta})\right|^2 ds+  \int_{\delta\lfloor \frac{t}{\delta}\rfloor}^{t}\left|{\sf h}(y_s)\right|^2 ds.
\end{aligned}
\end{equation}
Writing ${\sf h}(y_s)-{\sf h}(y_{i\delta})$ in the first term of the right hand side of the above equation as ${\sf h}(y_s)-D_2c(y_{i\delta},y_{i\delta})\cdot c(y_s,y_s)+ D_2c(y_{i\delta},y_{i\delta})\cdot c(y_s,y_s) -{\sf h}(y_{i\delta})$ and recalling the definition of $\hh(x)$, we get
\begin{equation*}
\left|{\sf h}(y_s)-{\sf h}(y_{i\delta}) \right|^2 \lesssim |D_2c(y_{i\delta},y_{i\delta})\{c(y_{i\delta},y_{i\delta})-c(y_s,y_s)\}|^2 
+ |\{D_2c(y_{i\delta},y_{i\delta})-D_2c(y_s,y_s)\}c(y_s,y_s)|^2.
\end{equation*} 
For $\jj(y,y)\triangleq D_2c(y,y):\BR^n \times \BR^n \to \BR^{n\times n},$ the entries of $\jj$ are the first order partial derivatives of $c(x,y)$ with respect to the second variable. Using Taylor's formula for each entry of $\jj$ (denoted by $\jj_{ij}$), we have 
\begin{equation}
\jj_{ij}(y_{i\delta},y_{i\delta})-\jj_{ij}(y_t,y_t)= \nabla \jj_{ij}(z_{ij})\cdot \{(y_{i\delta},y_{i\delta})-(y_t,y_t)\},
\end{equation}
for each $i,j=1,\cdots,n.$ Here, $z_{ij}\in \BR^n\times \BR^n$ are points lying on the line segment joining $(y_{i\delta},y_{i\delta})$ and $(y_t,y_t).$ Employing the boundedness of the second-order partial derivatives of $c$ (Assumption \ref{A:Derivative}), we have

\begin{equation*}
\begin{aligned}
|\jj(y_{i\delta},y_{i\delta})-\jj(y_t,y_t)|&=\max_{1\le j \le n}\sum_{i=1}^n|\jj_{ij}(y_{i\delta},y_{i\delta})-\jj_{ij}(y_t,y_t)|
\le \max_{1\le j \le n}\sum_{i=1}^n \|\nabla \jj_{ij}(z_{ij})\|\|(y_{i\delta},y_{i\delta})-(y_t,y_t)\|\\
& \lesssim |(y_{i\delta},y_{i\delta})-(y_t,y_t)|\lesssim |y_{i\delta}-y_t|. 
\end{aligned} 
\end{equation*}
Using the linear growth of $c(y_t,y_t)$, boundedness of $D_2c(y_t,y_t)$ and Lipschitz continuity of $c$ and Lemma \ref{L:Sampling-Difference-Levy}\color{black}, we get
\begin{equation}\label{E:h-Levy}
\left|{\sf h}(y_s)-{\sf h}(y_{i\delta}) \right|^2 \lesssim |y_{i\delta}-y_s|^2 \lesssim \delta^2.
\end{equation}
Now combining this estimate with equation \eqref{E:Rem-Eq} and recalling the definition of $R_t^{\eps,\delta}$ and equation \eqref{E:eps0}, we get the required estimate.
\end{proof}

\begin{proof}[Proof of Lemma \ref{L:M1-small}]
Recalling the definition of ${M}_1^{\eps,\delta}(t)$ and using the similar arguments to the proof of Lemma \ref{L:Remainder-Est-Levy} (in particular, equation \eqref{E:h-Levy}), we obtain the required estimate.
\end{proof}

\begin{proof}[Proof of Lemma \ref{L:M2-Est-Levy}]
Recalling the definition of ${ M}_2^{\eps,\delta}(t)$ and $\dd(t)$ from equations \eqref{E:M2-decomposition}  and  \eqref{E:ell-Levy} respectively, we have
\begin{multline}\label{E:M2-l}
{M}_2^{\eps,\delta}(t)-\dd(t) = \int_0^t D_2c(y_{\pi_\delta(s)}, y_{\pi_\delta(s)})\int_{\pi_\delta(s)}^s \frac{c(y_{\pi_\delta(r)}, y_{\pi_\delta(r)})}{\eps} \thinspace dr \thinspace ds\\
-\frac{\ell}{2}\int_0^t D_2c(y_s,y_s)\cdot c(y_s,y_s) \thinspace ds.
\end{multline}
Now, writing the integral $\int_0^t D_2c(y_{\pi_\delta(s)}, y_{\pi_\delta(s)})\int_{\pi_\delta(s)}^s {c(y_{\pi_\delta(r)}, y_{\pi_\delta(r)})} \thinspace dr \thinspace ds$ in the above equation \eqref{E:M2-l} as $\sum_{i=0}^{\lfloor \frac{t}{\delta}\rfloor-1}\int_{i\delta}^{(i+1)\delta}D_2c(y_{i\delta},y_{i\delta})\cdot c(y_{i\delta},y_{i\delta})(s-i \delta) \thinspace ds + \int_{\delta\lfloor \frac{t}{\delta}\rfloor}^{t} D_2c(y_{\pi_\delta (s)},y_{\pi_\delta (s)})\cdot c(y_{\pi_\delta (s)},y_{\pi_\delta (s)})(s-\pi_\delta(s)) \thinspace ds,$ we get

\begin{equation*}
\begin{aligned}
{M}_2^{\eps,\delta}(t)-\dd(t)&= (1/\eps)\sum_{i=0}^{\lfloor \frac{t}{\delta}\rfloor-1}\int_{i\delta}^{(i+1)\delta}D_2c(y_{i\delta},y_{i\delta})\cdot c(y_{i\delta},y_{i\delta})(s-i \delta) \thinspace ds  - \frac{\ell}{2}\int_0^t D_2c(y_s,y_s)\cdot c(y_s,y_s) \thinspace ds\\
&\qquad \qquad \qquad \qquad \qquad \quad +(1/\eps)\int_{\delta\lfloor \frac{t}{\delta}\rfloor}^{t} D_2c(y_{\pi_\delta (s)},y_{\pi_\delta (s)})\cdot c(y_{\pi_\delta (s)},y_{\pi_\delta (s)})(s-\pi_\delta(s)) \thinspace ds.
\end{aligned}
\end{equation*}

We note here that $\int_{i\delta}^{(i+1)\delta}D_2c(y_{i\delta},y_{i\delta})\cdot c(y_{i\delta},y_{i\delta})(s-i \delta) \thinspace ds=D_2c(y_{i\delta},y_{i\delta})\cdot c(y_{i\delta},y_{i\delta})({\delta^2}/{2}).$ Thus, taking norm on both sides of the above equation and then using \eqref{E:Triangle-type-ineq-Levy}, we obtain 
\begin{multline*}
\left|{M}_2^{\eps,\delta}(t)-\dd(t) \right|^2 \lesssim \left|\frac{\delta}{2\eps}\sum_{i=0}^{\lfloor \frac{t}{\delta}\rfloor-1}\delta D_2c(y_{i\delta},y_{i\delta})\cdot c(y_{i\delta},y_{i\delta}) \right. \\
 \left. \qquad \qquad \qquad \qquad \qquad \qquad \qquad -\frac{1}{2}\left(\ell -\frac{\delta}{\eps}+ \frac{\delta}{\eps} \right)\int_0^t D_2 c(y_s,y_s)\cdot c(y_s,y_s) \thinspace ds\right|^2\\
 +\delta^2 \left| D_2c(y_{\pi_\delta (t)},y_{\pi_\delta (t)})\cdot c(y_{\pi_\delta (t)},y_{\pi_\delta (t)}) \right|^2\frac{\delta^2}{\eps^2}\\
\lesssim \frac{1}{2^2}\frac{\delta^2}{\eps^2}\left|\sum_{i=0}^{\lfloor \frac{t}{\delta}\rfloor-1}\delta D_2c(y_{i\delta},y_{i\delta})\cdot c(y_{i\delta},y_{i\delta})-\int_0^t D_2 c(y_s,y_s)\cdot c(y_s,y_s) \thinspace ds\right|^2\\
+ \frac{1}{2^2}\left|\ell -\frac{\delta}{\eps} \right|^2 \left|\int_0^t D_2 c(y_s,y_s)\cdot c(y_s,y_s) \thinspace ds \right|^2 \\
+ \delta^2 \left| D_2c(y_{\pi_\delta (t)},y_{\pi_\delta (t)})\cdot c(y_{\pi_\delta (t)},y_{\pi_\delta (t)}) \right|^2\frac{\delta^2}{\eps^2}.
\end{multline*}

Now, using the liner growth of $c(y_s,y_s)$ and boundedness of $D_2c(y_s,y_s)$, we get
\begin{multline*}
\left|{M}_2^{\eps,\delta}(t)-\dd(t) \right|^2 
 \lesssim
\frac{\delta^2}{2^2\eps^2}\left|\sum_{i=0}^{\lfloor \frac{t}{\delta}\rfloor-1}\delta D_2c(y_{i\delta},y_{i\delta})\cdot c(y_{i\delta},y_{i\delta})-\int_0^t D_2 c(y_s,y_s)\cdot c(y_s,y_s) \thinspace ds\right|^2\\
+\left(\frac{\delta^{4}}{\eps^2}+\frac{\varkappa^2(\eps)}{2^2} \right)\left(1+\sup_{0\le t \le T}|y_t|^2 \right).
\end{multline*}

Next, putting the above estimate in equation \eqref{E:M2-l}, the required result follows by Lemmas \ref{L:Linear-Sys-Sampling-Est-Levy}, \ref{L:Remainder-Est-Levy} and equation \eqref{E:eps0}.
\end{proof}

\begin{proof}[Proof of Lemma \ref{L:M3-Est-WhiteNoise-Levy}]
For $t \in [0,T],$ recalling the definition of ${\sf M}_3^{\eps,\delta}(t)$ from equation \eqref{E:M-Rep-Levy} and using H$\ddot{\text{o}}$lder's inequality, we obtain
\begin{equation*}
|{\sf M}_3^{\eps,\delta}(t)|^2 \le \left(\int_0^t|D_2c(y_s,y_s)|^{2}\thinspace ds\right) \left(\int_0^t \left|\int_{\pi_{\delta}(s)}^{s}\sigma(Y_{u-}^{\eps,\delta}) \thinspace dW_u\right|^2\thinspace ds\right).
\end{equation*}
Taking supremum over $[0,T]$ followed by expectation and then using Assumption \ref{A:Derivative} for the boundedness of $D_2c(y_s,y_s)$, we have
\begin{equation}\label{E:M3-CLT-Levy-I}
\BE\left[\sup_{0\le t \le T}|{\sf M}_3^{\eps,\delta}(t)|^2\right]\lesssim
T \int_0^T \BE  \left|\int_{\pi_{\delta}(s)}^{s}\sigma(Y_{u-}^{\eps,\delta})\thinspace dW_u\right|^2 ds.
\end{equation}
Now, let $\sigma_i \in \BR^n$, $1 \le i \le n$, represent the columns of the matrix $\sigma$, then
\begin{equation*}
\begin{aligned}
\BE \left|\int_{\pi_{\delta}(s)}^{s}\sigma(Y_{u-}^{\eps,\delta})\thinspace dW_u\right|^2 
&=\BE \left|\sum_{i=1}^{n}\int_{\pi_{\delta}(s)}^{s}\sigma_i(Y_{u-}^{\eps,\delta})\thinspace dW_u^i\right|^2 
 \lesssim \BE \sum_{i=1}^{n}\left|\int_{\pi_{\delta}(s)}^{s}\sigma_i(Y_{u-}^{\eps,\delta})\thinspace dW_u^i \right|^2 \\
&\lesssim \sum_{i,j=1}^{n} \BE \left| \int_{\pi_{\delta}(s)}^{s}\sigma_{ji}(Y_{u-}^{\eps,\delta})\thinspace dW_u^i\right|^2 
 = \sum_{i,j=1}^{n}  \BE \int_{\pi_{\delta}(s)}^{s}\sigma_{ji}^2(Y_{u-}^{\eps,\delta})\thinspace du, 
\end{aligned}
\end{equation*}
where the last equality of the above equation follows by It\^{o} isometry for stochastic integrals. Next, using the linear growth property of $\sigma$ from Assumption \ref{A:Linear-growth-condition-Levy} in the above equation followed by a little algebra, we have 
\begin{equation}\label{E:M3-CLT-Levy-II}
\begin{aligned}
\BE  \left|\int_{\pi_{\delta}(s)}^{s}\sigma(Y_{u-}^{\eps,\delta})\thinspace dW_u\right|^2
& \lesssim \left(\int_{\pi_{\delta}(s)}^{s} \left( 1+ |y_u|^2 +\BE \left[\sup_{0 \le u \le T} |Y_u^{\eps,\delta}-y_u|^2 \right]  \right)du \right).
\end{aligned}
\end{equation}
Putting this last expression in equation \eqref{E:M3-CLT-Levy-I}, and then using Theorem \ref{T:LLN} and Lemma \ref{L:Linear-Sys-Sampling-Est-Levy}, we get the required estimate.
\end{proof}

\begin{proof}[Proof of Lemma \ref{L:M4-Est-PoissonNoise}]
Recalling the definition of ${\sf M}_4^{\eps,\delta}(t)$ from Lemma \ref{L:M-terms-levy} and using H$\ddot{\text{o}}$lder's inequality, we get 
$$|{\sf M}_4^{\eps,\delta}(t)|^2 \le \left(\int_0^t|D_2c(y_s,y_s)|^{2}\thinspace ds\right) \left(\int_0^t \left| \int_{\pi_\delta(s)}^s \int_{E}F(Y_{r-}^{\eps,\delta},x)\widetilde{N}(dr,dx)\right|^2 ds\right).$$ Taking supremum over $[0,T]$ followed by expectation and then using Assumption \ref{A:Derivative} for the boundedness of $D_2c(y_s,y_s)$, we have
\begin{equation}\label{E:M4-CLT-Levy-I}
\BE\left[\sup_{0\le t \le T}|{\sf M}_4^{\eps,\delta}(t)|^2 \right]\lesssim
T \int_0^T \BE  \left| \int_{\pi_\delta(s)}^s \int_{E}F(Y_{r-}^{\eps,\delta},x)\widetilde{N}(dr,dx)\right|^2  ds.
\end{equation}
Now,
\begin{equation*}
\begin{aligned}
\BE  \left| \int_{\pi_\delta(s)}^s \int_{E}F(Y_{r-}^{\eps,\delta},x)\widetilde{N}(dr,dx)\right|^2 & \lesssim \BE\left[  \sum_{i=1}^{n} \left| \int_{\pi_\delta(s)}^s \int_{E}F_i(Y_{r-}^{\eps,\delta},x)\widetilde{N}(dr,dx)\right|^2\right]\\
& =  \sum_{i=1}^{n} \BE \left| \int_{\pi_\delta(s)}^s \int_{E}F_i(Y_{r-}^{\eps,\delta},x)\widetilde{N}(dr,dx)\right|^2.
\end{aligned}
\end{equation*}
Using It\^{o} isometry for the right hand side of the above equation and then linear growth condition on $F$ from Assumption \ref{A:Linear-growth-condition-Levy} with $\nu(E)<\infty$, we have
\begin{equation}\label{E:M4-CLT-Levy-II}
\begin{aligned}
\BE \left| \int_{\pi_\delta(s)}^s \int_{E}F(Y_{r-}^{\eps,\delta},x)\widetilde{N}(dr,dx)\right|^2  & \lesssim \sum_{i=1}^{n} \BE\left[\int_{\pi_\delta(s)}^s \int_{E}|F_i(Y_{r-}^{\eps,\delta},x)|^2 \thinspace dr \thinspace \nu(dx) \right]\\
& \lesssim \sum_{i=1}^{n} \BE\left[\int_{\pi_\delta(s)}^s (1+ |Y_{r-}^{\eps,\delta}|^2)\thinspace dr\right].
\end{aligned}
\end{equation}
Now, combining the arguments of equation \eqref{E:M3-CLT-Levy-II} with equation \eqref{E:M4-CLT-Levy-I} and \eqref{E:M4-CLT-Levy-II}, we get the required result.
\end{proof}

\section{Conclusions and future research directions}\label{S:Conclusion-Levy}
In this paper, we studied the dynamics of a hybrid nonlinear system under the combined effect of fast periodic sampling and small random perturbations with jumps. The dynamics of stochastically perturbed system is found close to the idealized one.
Further, in the fluctuation study of rescaled process around its mean, the limiting {\sc sde}, interestingly, has an extra drift term capturing both the sampling and noise effect. The limiting {\sc sde} is found to vary, depending on the relative rates at which the two small parameters approach zero.

There are a couple of future research directions to extend the present work. Firstly, being motivated by the recent articles, for example, \cite{cerrai2009khasminskii,fu2011averaging,xu2015strong},  one can explore the present work in the context of stochastic partial differential equations ({\sc spde}), where the drift in our main equation \eqref{E:sde-intro} can be replaced by some suitable infinite dimensional operators. It is interesting to see here how does the classical Khasminskii, or the approaches presented in the articles \cite{fu2011averaging,xu2015strong}, are useful. Secondly, it is also interesting and a bit challenging to study the {\sc sde} \eqref{E:sde-intro} with the \textit{non-Lipschitz} drift and diffusion coefficients. For a non-Lipschitz case, recently, \cite{negrea2023pathwise} studied a pathwise uniqueness result for {\sc sde} with general drift and diffusion coefficients. So, checking whether the sufficient conditions provided in \cite{negrea2023pathwise} work in our sampling problem, may be a separate and follow-up project. Finally, we are also interested to explore the asymptotic analysis of rare events in terms of certain rate functions, so-called {\sc ldp}.

\subsection*{Acknowledgments}
The author would like to thank Prof. Chetan D. Pahlajani for his helpful discussion throughout the manuscript preparation and the referees for their helpful comments and suggestions. I also want to thank the Discipline of Mathematics, {\sc iit} Gandhinagar where this work was completed as a part of my Ph.D. thesis.

\bibliographystyle{alpha}
\bibliography{references}

\end{document}